\documentclass{article}
\usepackage{url}
\usepackage{graphicx}
\graphicspath{{}}
\usepackage[tbtags]{amsmath}
\usepackage{amsfonts}
\usepackage{bbold}
\usepackage{mathtools}
\mathtoolsset{showonlyrefs=fasle,showmanualtags}
\usepackage{amsthm} 
\theoremstyle{plain}
\newtheorem{theorem}{Theorem}[section]
\newtheorem{lem}[theorem]{Lemma}

\theoremstyle{definition}

\theoremstyle{remark}

\usepackage{amssymb}
\usepackage{algorithm}
\usepackage{algorithmic}
\usepackage{float}

\usepackage[english]{babel}
\usepackage{enumitem}
\usepackage{color,booktabs,colortbl}
\usepackage{multirow}
\usepackage{latexsym,bm}
\usepackage{multicol}
\usepackage{comment}
\usepackage{geometry}
\usepackage{fancyhdr}
\usepackage[colorlinks,hyperindex,bookmarks=true]{hyperref}
\definecolor{dark-red}{rgb}{0.4,0.15,0.15}
\definecolor{dark-blue}{rgb}{0.15,0.15,0.4}
\definecolor{medium-blue}{rgb}{0,0,0.5}
\hypersetup{%
  linkcolor = dark-red,%
  citecolor = dark-blue,%
  urlcolor = medium-blue,
  bookmarksnumbered=true, bookmarksopen=true, bookmarksopenlevel=2,
  pdftitle=paper for phase retrieval,
  pdfauthor=MathLee,
  pdfsubject= Mathematic,
  pdfkeywords={}}
\usepackage{caption}
\usepackage{subcaption}
\newcommand{\norm}[1]{\left\lVert#1\right\rVert}

\DeclareMathOperator*{\re}{Re}

\DeclareMathOperator*{\trace}{Tr}
\DeclareMathOperator*{\rank}{rank}

\usepackage{authblk}
 \usepackage{cite}

\begin{document} 
\title{Scalable Incremental Nonconvex Optimization Approach for Phase Retrieval}
\author[*]{Ji Li\thanks{email: keelee@csrc.ac.cn}}
\author[**]{Jian-Feng Cai\thanks{email: jfcai@ust.hk}}
\author[***]{Hongkai Zhao\thanks{email: zhao@math.uci.edu}}
\affil[*]{Beijing Computational Science Research Center, Beijing, China}
\affil[**]{Department of Mathematics, Hong Kong University of Science
and Technology, Clear Water Bay, Kowloon, Hong Kong}
\affil[***]{Department of Mathematics, University of California, Irvine, CA, USA}
\date{\today}
\maketitle
\begin{abstract}
We aim to find a solution $\bm{x}\in\mathbb{C}^n$ to a system of quadratic equations of the form
$b_i=\lvert\bm{a}_i^*\bm{x}\rvert^2$,
$i=1,2,\ldots,m$, e.g., the well-known NP-hard phase retrieval problem. As opposed to recently proposed state-of-the-art nonconvex methods, we revert to the semidefinite relaxation (SDR) PhaseLift convex formulation and propose a successive and incremental nonconvex optimization algorithm, termed as \texttt{IncrePR}, to indirectly minimize the resulting convex problem on the cone of positive semidefinite matrices. Our proposed method overcomes the excessive computational cost of typical SDP solvers as well as  the  need of a good initialization for typical nonconvex methods. For Gaussian measurements, which is usually needed for provable convergence of nonconvex methods, \texttt{IncrePR} with restart strategy outperforms state-of-the-art nonconvex solvers with a sharper phase transition of perfect recovery and typical convex solvers in terms of computational cost and storage. For more challenging structured (non-Gaussian) measurements often occurred in real applications, such as transmission matrix and oversampling Fourier transform, \texttt{IncrePR} with several restarts can be used to find a good initial guess. With further refinement by local nonconvex solvers, one can achieve a better solution than that obtained by applying nonconvex solvers directly when the number of measurements is relatively small. Extensive numerical tests are performed to demonstrate the effectiveness of the proposed method.
\end{abstract}
 
\section{Introduction}
\label{sec:introduction}
\subsection{The problem}

Consider the following $m$ quadratic equations for $\bm{x}\in\mathbb{C}^n$
\begin{equation}
  \label{eq:1}
  b_i=\lvert \bm{a}_i^*\bm{x}\rvert^2, \quad 1\leq i\leq m,
\end{equation}
where the measurements $\bm{b}=[b_1,\ldots,b_m]^T$ and
the design/sampling vectors $\{\bm{a}_i\}_{i=1}^m\in\mathbb{C}^n$ are known.  Having information
about $\lvert \bm{a}_i^*\bm{x}\rvert$ means that the
signs or phases of $\bm{a}_i^*\bm{x}$ are
missing. In general, problem~\eqref{eq:1} constitutes an instance of NP-hard nonconvex quadratic programming~\cite{Chen2015b}. On the other hand, if the missing phases of the
measurements $\{b_i\}_{i=1}^m$ are recovered, the true solution, denoted by $\bm{x}^{\natural}$ in this paper, can be
reconstructed by solving a system of linear equations. Problem~\eqref{eq:1} is referred to the well-known (Fourier) phase retrieval problem when $\bm{a}_i$'s are 
Fourier vectors. The phase retrieval problem is of paramount
importance in many fields of physical sciences and engineering, e.g., X-ray crystallography~\cite{Millane1990},
electron microscopy~\cite{Misell1973}, X-ray diffraction
imaging~\cite{Shechtman2015}, optics~\cite{Kuznetsova1988} and
astronomy~\cite{Fienup1982} etc. In these applications,
often one can only record intensity of the Fourier transform of a complex
signal due to physical
limitations of detectors, such as charge-coupled device (CCD)
cameras. 

In spite of its simple form and practical importance across
various fields, solving problem~\eqref{eq:1} is challenging. 
Even checking the feasibility is an NP-hard problem. Much effort has recently been devoted to determining the number of such equations $m$ necessary and/or sufficient for the uniqueness of the solution up to a global phase constant. It has been shown that, for the real case, $m=2n-1$ measurements with generic\footnote{For the definition of generic vectors, see the reference~\cite{Balan2015,Balan2006}.} vectors $\bm{a}_i$'s are sufficient and necessary for the uniqueness. For the complex case, the necessary condition for uniqueness requires $m\geq 4n-4$ measurements~\cite{Balan2015,Balan2006}. Assuming that the system~\eqref{eq:1} admits an unique
solution $\bm{x}^{\natural}$ up to a phase constant, one needs an efficient and robust numerical algorithm to compute it.  However, current state-of-the-art methods are either not well scaled with respect to the dimension of the signal, $n$, in terms of computational complexity or can only guarantee local convergence when the number of measurements $m$ is small. Our objective is to develop a conceptually simple while numerically efficient and stable algorithm to 
find the solution. 

\subsection{Prior work}
\label{sec:relatived-work}

To develop provable algorithms and facilitate the analysis of phase retrieval,
most recent works assume Gaussian random measurements~\cite{Candes2015,Chen2015b,Candes2012,Wang16,Netrapalli2013},  which are not very realistic in most applications. The measurement setup is said to be Gaussian model if the sampling vector
$\bm{a}_i\sim\mathcal{N}(0,I)$ and
$\bm{a}_i\sim\mathcal{CN}(0,I)=\mathcal{N}(0,I/2)+i\mathcal{N}(0,I/2)$ for the real and complex cases respectively. The randomness assumption is the key ingredient for the success of those proposed algorithms when the number of measurements is large enough. The only ambiguity of a solution to~\eqref{eq:1} for Gaussian model is a global phase shift, i.e., $e^{i\theta}\bm{x}^{\natural}$ for some $\theta$. For the classical oversampling Fourier phase retrieval, the ambiguity of solution becomes worse, which includes phase shift, translation shift and mirror image~\cite{Hayes1982}. To mitigate the more severe ambiguity issue, redundant measurements are introduced, such as measurements with
masks~\cite{Fannjiang2012a,Candes2013}, short-time Fourier
transform~\cite{Jaganathan2015,Qian2014}, and ptychography~\cite{Qian2014}. In
practice, it is typically assumed that problem~\eqref{eq:1} admits an unique solution up to a global phase constant as long as the number of measurements is large enough. For most practical measurements, such as transmission matrix in the phase package~\cite{Chandra2017PhasePack} and Fourier transform, both the theory and those algorithms based on Gaussian model do not work as well as for Gaussian model. 

Much effort has recently been devoted to devising provable phase retrieval solvers for Gaussian model~\cite{Candes2015,Candes2012,Wang16,Wang2017Solving,Wang2017SPARTA}. These algorithms fall into two realms: convex and nonconvex ones. Convex
approaches rely on the so-called lifting technique to turn the
quadratic nonlinear constraints of $\bm{x}$ into linear constraints of the
matrix $\bm{X}=\bm{x}\bm{x}^*$. Trace regularized semidefinite relaxation (SDR) formulation, termed as PhaseLift, can be designed~\cite{Candes2012}, where low rank promoting trace regularization replaces the rank one constraint. Unlike nonconvex ones, the solution to SDR formulation is independent of initialization. However, such algorithms typically involve storing an $n\times n$ matrix variable and solving an SDP (semidefinite programming) problem, which suffer from prohibitive computation cost and storage for large scale problems. Another variant, PhaseCut, which is based on a different SDR reformulation, has the same scaling problem~\cite{Fogel2013}. We say that the SDR for phase retrieval~\eqref{eq:1} is \emph{tight}, if the optimal solution $\bm{X}^{\natural}$ of SDR is unique and of rank one, i.e., $\bm{X}^{\natural}=\bm{x}^{\natural}{\bm{x}^{\natural}}^*$. Then $\bm{x}^{\natural}$ is the solution to~\eqref{eq:1} up to a phase shift. The tightness of PhaseLift with trace minimization for Gaussian model requires the number of measurements $m$ on the order of $n$~\cite{Candes2013b}. Another recent convex relaxation model~\cite{Dhifallah2017Phase,Demanet2013} reformulated~\eqref{eq:1} as a linear programming problem in the natural parameter vector domain $\mathbb{C}^n$. However, its performance depends on a good anchor vector, which has to form a small enough angle with the true solution and serves as a similar role as a good initial guess for nonconvex approaches to~\eqref{eq:1}.

On the other hand, nonconvex approaches directly deal with~\eqref{eq:1} in the natural parameter space $\mathbb{C}^n$. Based upon different formulations, recently proposed nonconvex approaches include:
alternating projection methods, such as AltMinPhase~\cite{Netrapalli2013} and Kaczmarz variant~\cite{Wei2015}; gradient flow based methods, such as
Wirtinger flow (WF) and its truncated version
(TWF)~\cite{Candes2015,Chen2015b}, truncated amplitude flow (TAF)~\cite{Wang16} and reweighted amplitude flow (RAF)~\cite{Wang2017Solving}; Wirtinger flow (WF) with an activation function~\cite{Lizhenzhen}; message passing based methods, such as approximate message passing for amplitude-based optimization (AMP.A)~\cite{Ma2018Approximate}; and the prox-linear procedure via composite optimization~\cite{Duchi2017Solving}. These approaches lead to significant computational advantage over the lifting based SDR counterparts. However, due to the existence of possibly many stationary points, nonconvex approaches in general are computationally intractable to find the true solution $\bm{x}^{\natural}$ from an arbitrary initial guess. How to obtain an initialization that is good enough for these nonconvex methods is the key ingredient for a provable algorithm with convergence guarantee. 
In a nutshell, for nonconvex approaches for Gaussian model, there are generally two stages involved to obtain a global minimizer:
\begin{enumerate}
\item The initialization stage: for Gaussian model, a good initial guess can be computed efficiently by some routines, such as the spectral method~\cite{Netrapalli2013,Candes2015}, the truncated spectral method~\cite{Chen2015b} and their variants, as well as the orthogonality-promoting method~\cite{Wang16} (a.k.a. null method~\cite{Chen2015a}) and the (re-weighted) maximal correlation method~\cite{Wang2017Solving}. The obtained initial guess is proved to be good enough, provided the number of measurements $m$ is on the order of $n$. 
\item The refinement stage: a refinement scheme, generally a gradient descent algorithm, is then used to push the initial guess toward to the global minimizer. The convergence to global minimizer is shown through analysis of the population behavior of different objectives with the help of Gaussian model assumption~\cite{Candes2015,Chen2015b,Wang16,Wang2017Solving}. 
\end{enumerate}

For those state-of-the-art provable algorithms for phase retrieval, exact recovery  for Gaussian model is guaranteed when the number of measurements $m=cn\log n$~\cite{Candes2015,Demanet2013} or $m=cn$~\cite{Candes2013b,Wang16,Wang2017Solving,Chen2015b}, where $c$ is often a fixed but rather large constant independent of $n$. The reported smallest numerical constant $c$ for real case among the existing algorithms is $2$ for RAF~\cite{Wang2017Solving}. It is shown recently that for Gaussian model the intensity-based least-squares objective (used by WF) and its truncated modification with an activation function admits a benign geometric structure without spurious local minima from $m=\mathcal{O}(n\log^3n)$~\cite{Sun2016,Li16} and $m=\mathcal{O}(n)$~\cite{Lizhenzhen} measurements respectively. However, the number of measurements $m=cn$ needs a large constant $c$  far greater than $2$ for real case. The required number of measurements for good recovery is even larger for non-Gaussian measurements, e.g., for ptychography~\cite{Qian2014}. 

When the number of measurements is large enough, a good initialization may be obtained and in this regime,  those provable nonconvex algorithms often work well. For applications where the number of measurements is limited, a good initialization is difficult to obtain and nonconvex approaches usually fail to locate a satisfying solution. For example, due to the significant bias of the statistical properties of finite samples of Gaussian vectors, the initial guess becomes worse when $m$ becomes smaller.
Moreover, most initialization strategies only work for Gaussian model and fails for non-Gaussian measurements in most practices. The quality of initialization also downgrades with increased noise in the measurements. Without the help of a good initial guess, nonconvex approaches stagnate at local stationary points, which are far away from the true solution. 

\subsection{Contribution}
\label{sec:paper}

Motivated by the global convergence of SDR PhaseLift approach, we revert to PhaseLift formulation of~\eqref{eq:1}. Its tightness is proved for Gaussian model for which the rank one solution to PhaseLift provides the true solution to~\eqref{eq:1}. To avoid the computation cost and storage of a typical SDP solver for PhaseLift, we propose a simple, well-scaled method to find the solution to PhaseLift by the decomposition of a Hermitian matrix $\bm{X}$ into $\bm{Y}\bm{Y}^*$, where $\bm{Y}\in\mathbb{C}^{n\times p}$. Our method, termed as \texttt{IncrePR}, indirectly solves PhaseLift by a sequence of nonconvex optimizations in space $\mathbb{C}^{n\times p}$ with incremental column number $p$. Our incremental nonconvex optimization is similar to the approach for SDP proposed in literature~\cite{Journ2010Low}. However, it is the first time applied to solving the phase retrieval problem with the extension to Hermitian matrix. For Gaussian model, to pursue sharper \emph{phase transition} for noiseless measurements, where \emph{phase transition} is defined as the numerical limit of ratio $m/n$ to successfully recover $\bm{x}^{\natural}$ from equations~\eqref{eq:1}, we propose to solve three PhaseLift problems involving different parameters with a restart strategy. \texttt{IncrePR} with restart shows the sharpest phase transition among the state-of-the-art solvers for both real and complex cases. It achieves perfect recovery if the number of measurements $m\geq 1.9n$ and $m\geq 2.7n$ for real and complex cases respectively. It beats the theoretical guarantee of $m=2n-1$~\cite{Balan2006} and $m\geq 4n-4$~\cite{Balan2015} for generic real and complex cases respectively. The extra decreasing of $m/n$ is due to the fact that Gaussian vector belongs to generic vector, so Gaussian model exhibits smaller phase transition than the above theoretical guarantee for generic vector measurements. For noisy measurements, \texttt{IncrePR} with restart locates good approximate solution with the help of global convergence and stability of PhaseLift~\cite{Candes2012}.

Although the tightness of PhaseLift is no longer guaranteed for other type of more practical measurements, the rank one approximation of the solution to PhaseLift is generally close to the true solution to~\eqref{eq:1}. We can then feed it to nonconvex approaches to recover a much better solution than that of direct nonconvex approaches from a random initialization. In other words,  \texttt{IncrePR} with restart can be used as a good initializer. Numerical tests show that \texttt{IncrePR} with restart outperforms the state-of-the-art solver, reweighted amplitude flow (RAF)~\cite{Wang2017Solving}, for transmission matrix measurements (available in a public phase package~\cite{Chandra2017PhasePack}) when the number of measurements is small. For more difficult structured oversampling Fourier phase retrieval, \texttt{IncrePR} with restart is comparable to the hybrid input-output (HIO) projection-based method, while all other methods built for Gaussian model fail to compute a meaningful reconstruction~\cite{Candes2013} when the number of measurements is not large enough. 

\subsection{Other relevant literature}

The decomposition-based nonconvex optimization idea has been applied to solving linear cost function~\cite{Burer2003A} and nonlinear cost function~\cite{Journ2010Low} with affine constraints on the cone of positive semidefinite matrices. It has also been extended to general (asymmetrical) low rank matrix recovery problems~\cite{Mishra2013Low,Ciliberto2017Reexamining}. The first attempt to apply decomposition $\bm{X}=\bm{Y}\bm{Y}^*$ within a fixed matrix dimension $\mathbb{C}^{n\times p}$ to solve PhaseLift~\cite{Candes2013} was proposed in~\cite{WenHuang}, where the authors proposed a dynamic ranking decreasing strategy. In this work, we employ the incremental strategy and propose a post-verifiable condition for convergence and termination by establishing a relation between the stationary solutions to the nonconvex and convex problems respectively. Compared to the fixed dimension strategy, \texttt{IncrePR} performs better and terminates quickly with very few incremental steps in our experiments.

To develop a scalable nonconvex algorithm, authors in~\cite{Yurtsever2017Sketchy} considered applying the conditional gradient method to another SDR problem of~\eqref{eq:1} with constraint on the maximum trace of $\bm{X}$. To avoid updating of the full matrix $\bm{X}$, they derived a novel method, termed as SketchyCGM, driven by the measurement vectors and sketchy of the underlying low rank matrix during iterations. Compared to the standard PhaseLift solver, the algorithm outperforms when the number of measurements $m$ is large enough, and underperforms when $m$ is close to the phase transition point~\cite{Chandra2017PhasePack}.

\subsection{Notation and organization}

We use $\mathbb{H}^n$ and $\mathbb{H}_+^n$ to denote the set of complex and positive semidefinite $n\times n$ Hermitian matrices respectively. $\langle \bm{U},\bm{V}\rangle=\re(\trace(\bm{U}^*\bm{V}))$ denotes the standard inner product in $\mathbb{H}^n$, where $\{\cdot\}^*$ is the conjugated transpose. We use bold font to denote vectors and bold capital font for matrices. $\bm{a}\circ\bm{b}$ and $\bm{a}/\bm{b}$ denote the entrywise multiplication and division between vectors $\bm{a}$ and $\bm{b}$, respectively.
The organization of the paper is as follows. We first review the PhaseLift formulation and cite known results of exact and stable recovery for noiseless and noisy Gaussian models for completeness. Then we describe our incremental nonconvex approach (\texttt{IncrePR}) for solving the PhaseLift problem. 
Combined with a restart strategy, \texttt{IncrePR} can achieve optimal phase transition for Gaussian model or obtain a good initialization for non-Gaussian measurements. Various numerical tests and comparisons are presented to demonstrate the performance of our proposed method. Conclusion is presented at the end.  

\section{PhaseLift: a review}
\label{sec:algor-anoth-trunc}

Note that our proposed method, \texttt{IncrePR}, works for both real and complex cases. For generality, we formulate the problem and the algorithm for the complex case. 

The problem~\eqref{eq:1} is to find a vector $\bm{x}\in\mathbb{C}^n$ subject to the system of $m$ quadratic equations $\lvert \bm{a}_i^*\bm{x}\rvert^2=b_i,i=1,2,\ldots,m$. It can be viewed as an instance of quadratic constraint quadratic programming (QCQP). Following the classical Schor's relaxation~\cite{Luo2010Semidefinite}, the vector variable $\bm{x}$ can be lifted into a matrix variable $\bm{X}=\bm{x}\bm{x}^*$~\cite{Candes2013}. The nonlinear quadratic constraints for $\bm{x}$ can be cast as linear constraints of $\bm{X}$, i.e.,
\begin{equation*}
  \lvert\bm{a}_i^*\bm{x}\rvert^2=\trace(\bm{A}_i\bm{X})=b_i, \quad \bm{A}_i=\bm{a}_i\bm{a}_i^*.
\end{equation*}
Thus problem~\eqref{eq:1} can be formulated as the following matrix problem
\begin{equation*}
  \begin{aligned}
    \text{find}\quad &\bm{X}\in\mathbb{C}^{n\times n}\\
    \text{s.t.}\quad &\trace(\bm{A}_i\bm{X})=b_i,i=1,2,\ldots,m,\\
    & \bm{X}\succeq 0,\quad \rank(\bm{X})=1.
  \end{aligned}
\end{equation*}
 The problem is still NP-hard due to nonconvexity of the rank one constraint. It can be equivalently transformed into a rank minimization problem:
\begin{equation*}
  \begin{aligned}
    \min_{\bm{X}}\quad & \rank(\bm{X})\\
    \text{s.t.}\quad & \trace(\bm{A}_i\bm{X})=b_i,i=1,2,\ldots,m,\\
    & \bm{X}\succeq 0.
  \end{aligned}
\end{equation*}
When the rank minimization is further relaxed to trace minimization, the new problem is of the following form and called semidefinite relaxation (SDR) or PhaseLift,
\begin{equation}
  \label{eq:sdp}
  \begin{aligned}
    \min_{\bm{X}}\quad & \trace(\bm{X})\\
    \text{s.t.} \quad & \trace(\bm{A}_i\bm{X})=b_i,i=1,2,\ldots,m,\\
    & \bm{X}\succeq 0.
  \end{aligned}
\end{equation}

Cand\`{e}s et al showed that the solution to problem~\eqref{eq:sdp} is of rank one and its rank one decomposition factor is exactly the solution to problem~\eqref{eq:1} with high probability for Gaussian model, provided the number of measurements $m=\mathcal{O}(n)$~\cite{Candes2013b}. Furthermore, the solution is also stable for noisy measurements, no matter what distribution the noise is drawn from. The trace minimization in~\eqref{eq:sdp} is unnecessary in some cases. The most obvious case is that the intensity measurements determine the trace of $\bm{X}$, such as the Fourier phase retrieval and the case that the identity matrix is in the span of $\{\bm{A}_i\}_{i=1}^m$. When $m=\mathcal{O}(n\log n)$ for Gaussian model, the trace minimization is unnecessary~\cite{Demanet2013}. Note that there exists a gap for $m$ between $\mathcal{O}(n)$ and $\mathcal{O}(n\log n)$, which is the critical transition of the tightness between the trace minimization problem~\eqref{eq:sdp} and the set feasibility problem, i.e., problem~\eqref{eq:sdp} without trace minimization. The available theoretical tightness and phase transition results are for Gaussian model.  For other type of measurements, the optimal solution to~\eqref{eq:sdp} may not be of rank one while its rank one approximation is often close to the rank one matrix $\bm{x}^{\natural}{\bm{x}^{\natural}}^*$.  

Due to possible noise in measurements, Cand\`{e}s et al~\cite{Candes2013} proposed a first-order scheme to solve the trace regularized SDP problem:
\begin{equation}
\label{eq:phaselift}
\begin{aligned}
  \min_{\bm{X}} \quad & f_{\lambda}(\bm{X})=f_0(\bm{X})+\lambda\trace(\bm{X}),\\
  \text{s.t.}\quad & \bm{X}\succeq 0,
\end{aligned}
\end{equation}
where the parameter $\lambda$ trades off between the data fidelity term $f_0(\bm{X})$ and the low rank promoting trace term. To distinguish the two problems, we call~\eqref{eq:sdp} PhaseLift with trace minimization and~\eqref{eq:phaselift} PhaseLift with trace regularization. The data fidelity term $f_0(\bm{X})$ depends on the noise model. For least-squares loss, the objective function in matrix variable $\bm{X}$ is 
\begin{equation}
\label{eq:func}
  f_0(\bm{X}) = \frac{1}{2m}\norm{\mathcal{A}(\bm{X})-\bm{b}}^2.
\end{equation}
For Poisson likelihood estimation, the objective function is
\begin{equation}
  \label{eq:pfunc}
  f_0(\bm{X})=\langle \bm{1},\mathcal{A}(\bm{X})-\bm{b}\circ\log(\mathcal{A}(\bm{X}))\rangle=\sum_{i=1}^m\trace(\bm{A}_i\bm{X})-b_i\log(\trace(\bm{A}_i\bm{X})).
\end{equation}
The linear operator in~\eqref{eq:func} and~\eqref{eq:pfunc} is defined as 
\begin{equation*}
\mathcal{A}:\mathbb{H}^n\mapsto \mathbb{R}^m, \mathcal{A}(\bm{X})=[\trace(\bm{A}_1\bm{X}), \trace(\bm{A}_2\bm{X}), \cdots, \trace(\bm{A}_m\bm{X})]^T.
\end{equation*}
And its adjoint operator is denoted as $\mathcal{A}^*:\mathbb{R}^m\mapsto\mathbb{H}^n, \mathcal{A}^*(\bm{b})=\sum_{i=1}^mb_i\bm{A}_i$. 

Problem~\eqref{eq:phaselift} was solved by first-order method with matrix variable $\bm{X}$ in~\cite{Candes2013}, whereas the storage of $\bm{X}\in\mathbb{H}^n$ and the projection onto $\mathbb{H}_+^n$ are intractable for large-scale problems. Motivated by decomposition of variable $\bm{X}$ into $\bm{Y}\bm{Y}^*$, where $\bm{Y}\in\mathbb{C}^{n\times p}$ and $ p\ll n$ due to the low rank property of $\bm{X}$, we propose an incremental nonconvex optimization approach to indirectly solve~\eqref{eq:phaselift} with significantly reduced computational cost and storage.

\section{Nonconvex solver for~\eqref{eq:phaselift}}
\label{sec:meta}

With the choices of objectives~\eqref{eq:func} or~\eqref{eq:pfunc}, problem~\eqref{eq:phaselift} is a smooth convex optimization problem in the cone of $\mathbb{H}_+^n$. We substitute expression $\bm{X}=\bm{Y}\bm{Y}^*$ with $\bm{Y}\in \mathbb{C}^{n\times p}$ into~\eqref{eq:phaselift}, then the dimension of optimization variable decreases from $n^2$ to $np$. And the positive semidefiniteness of $\bm{X}$ is automatically satisfied by the decomposition.

The resulting nonconvex problem in $\bm{Y}$ is 
\begin{equation}
  \label{eq:nonconvex}
  \min_{\bm{Y}\in\mathbb{C}^{n\times p}}\quad \tilde{f}_p^{\lambda}(\bm{Y})=f_0(\bm{Y}\bm{Y}^*)+\lambda\norm{\bm{Y}}_F^2.
\end{equation}
If $p=1$, \eqref{eq:nonconvex} becomes the $\ell_2$ regularized Wirtinger flow (a.k.a. intensity-based least-squares) in the natural parameter space $\mathbb{C}^n$. 

The computation storage and cost of solving~\eqref{eq:nonconvex} is significantly less than solving~\eqref{eq:phaselift} when $p\ll n$. 
Once $p$ is fixed, minimizing $\tilde{f}_p^{\lambda}(\bm{Y})$ will locate a stationary point. The question is how to change $p$ appropriately to help us to find the true solution to~\eqref{eq:phaselift} or a good approximation of it. The answer can be obtained from the analysis of the optimality conditions of the two associated problems.   

\subsection{Connection between the convex and nonconvex formulations}
\label{sec:phase-retrieval-via}

The possibility of indirectly solving problem~\eqref{eq:phaselift} by solving~\eqref{eq:nonconvex} is implied by the relation between their first and second order optimality conditions. For notation simplicity, we omit the script $\lambda$ and $p$ in $f_{\lambda}(\bm{X})$ and $\tilde{f}_p^{\lambda}(\bm{Y})$ when $\lambda$ and $p$ are fixed. Given a function $f:\mathbb{H}^n\to\mathbb{R}:\bm{X}\mapsto f(\bm{X})$, we define the function 
\begin{equation*}
  \tilde{f}:\mathbb{C}^{n\times p}\to \mathbb{R}: \bm{Y}\mapsto\tilde{f}(\bm{Y})=f(\bm{Y}\bm{Y}^*).
\end{equation*}
For a differentiable function $f$, the notation $\nabla_{\bm{X}}f(\bm{X}_0)$ refers to the gradient of $f$ at $\bm{X}_0$ with respect to the variable $\bm{X}$ in the Wirtinger sense~\cite{Candes2015,Kreutz-Delgado2009}, since here we deal with the real-valued function of complex variables, i.e.,
\begin{equation*}
  [\nabla_{\bm{X}} f(\bm{X}_0)]_{i,j}=\frac{\partial f}{\partial \overline{X}_{i,j}}(\bm{X}_0).
\end{equation*}
In the real case, the gradient definition becomes the standard one. 

The directional derivative of $f$ at $\bm{X}_0$ in a direction $\bm{Z}\in \mathbb{H}^n$ is defined as
\begin{equation*}
  D_{\bm{X}}f(\bm{X}_0)[\bm{Z}]=\lim_{t\to 0}\frac{f(\bm{X}_0+t\bm{Z})-f(\bm{X}_0)}{t}.
\end{equation*}
If $f$ is differentiable, we have
\begin{equation*}
  D_{\bm{X}}f(\bm{X}_0)[\bm{Z}]=\langle \nabla_{\bm{X}} f(\bm{X}_0),\bm{Z}\rangle,
\end{equation*}
and
\begin{equation*}
  \nabla_{\bm{Y}} \tilde{f}(\bm{Y})= 2\nabla_{\bm{X}} f(\bm{Y}\bm{Y}^*)\bm{Y}, \quad
  D_{\bm{Y}}\tilde{f}(\bm{Y})[\bm{Z}]=\langle 2\nabla_{\bm{X}} f(\bm{Y}\bm{Y}^*)\bm{Y}, \bm{Z} \rangle, \quad \forall \bm{Y}, \bm{Z}\in\mathbb{C}^{n\times p}.
\end{equation*}

We state the first-order KKT conditions for problems~\eqref{eq:phaselift} and~\eqref{eq:nonconvex} in the following lemmas.
\begin{lem}
\label{lem:1}
  A global minimizer of problem~\eqref{eq:phaselift} is a Hermitian matrix $\bm{X}\in\mathbb{H}^n$ such that the first-order optimality conditions hold:
  \begin{equation}
    \label{eq:3}
    \begin{aligned}
      (\nabla_{\bm{X}} f_0(\bm{X})+\lambda \bm{I})\bm{X}&=0\\
      \nabla_{\bm{X}} f_0(\bm{X})+\lambda \bm{I}&\succeq 0\\
      \bm{X}&\succeq 0.
    \end{aligned}
  \end{equation}
The optimality conditions are sufficient and necessary for our convex optimization~\eqref{eq:phaselift}. 
\end{lem}

\begin{lem}
\label{lem:2}
  If $\bm{Y}$ is a stationary point of problem~\eqref{eq:nonconvex}, then it holds that
  \begin{equation*}
    \left(\nabla_{\bm{X}} f_0(\bm{Y}\bm{Y}^*)+\lambda \bm{I}\right)\bm{Y} = 0.
  \end{equation*}
\end{lem}

Comparing Lemma~\ref{lem:1} and~\ref{lem:2}, we have the following relation between the stationary solutions of problems~\eqref{eq:phaselift} and~\eqref{eq:nonconvex}. 
\begin{theorem}
\label{thm:1}
A stationary point $\bm{Y}$ of the nonconvex problem~\eqref{eq:nonconvex} provides a global minimizer $\bm{Y}\bm{Y}^*$ of problem~\eqref{eq:phaselift}, if the matrix
\begin{equation}
  \label{eq:2}
  \bm{S}_{\bm{Y}} = \nabla_{\bm{X}} f_0(\bm{Y}\bm{Y}^*) + \lambda \bm{I}
\end{equation}
is positive semidefinite.
\end{theorem}

The stationary point $\bm{Y}$ of problem~\eqref{eq:nonconvex} may be a saddle point. Here we state the second-order optimality condition for the unconstrained problem.
\begin{lem}
  For a local minimizer $\bm{Y}\in\mathbb{C}^{n\times p}$ of~\eqref{eq:nonconvex}, it holds that
  \begin{equation*}
    \trace(\bm{Z}^*D_{\bm{Y}}\left(\nabla_{\bm{Y}} (f_0(\bm{Y}\bm{Y}^*)+\lambda \norm{\bm{Y}}_F^2)\right)[\bm{Z}])\geq 0
  \end{equation*}
for any matrix $\bm{Z}\in\mathbb{C}^{n\times p}$.
\end{lem}
\begin{lem}
\label{lem:sec}
  For any matrix $\bm{Z}\in\mathbb{C}^{n\times p}$ such that $\bm{Y}\bm{Z}^*=0$, the following equality holds:
  \begin{equation*}
    \frac{1}{2}\trace(\bm{Z}^*D_{\bm{Y}}\left(\nabla_{\bm{Y}} (f_0(\bm{Y}\bm{Y}^*)+\lambda \norm{\bm{Y}}_F^2)\right)[\bm{Z}])=\trace(\bm{Z}^*\bm{S}_{\bm{Y}}\bm{Z}).
  \end{equation*}
\end{lem}

\begin{proof}
  According to the definition of directional derivative, we have
  \begin{equation*}
\begin{aligned}
& \frac{1}{2}\trace\left(\bm{Z}^*D_{\bm{Y}}\left(\nabla_{\bm{Y}} (f_0(\bm{Y}\bm{Y}^*)+\lambda \norm{\bm{Y}}_F^2)\right)[\bm{Z}]\right)  \\
=&\trace(\bm{Z}^*(\nabla_{\bm{X}} f_0(\bm{Y}\bm{Y}^*) + \lambda \bm{I})\bm{Z})+\trace(D_{\bm{Y}}\left(\nabla_{\bm{X}} f_0(\bm{Y}\bm{Y}^*) + \lambda \bm{I}\right)[\bm{Z}]\bm{Y}\bm{Z}^*)\\
=&\trace(\bm{Z}^*\bm{S}_{\bm{Y}}\bm{Z}),
\end{aligned}
\end{equation*}
where we use the equality $\trace(\bm{A}\bm{B})=\trace(\bm{B}\bm{A})$.
\end{proof}

The natural question is how to choose the column number $p$ of $\bm{Y}$. We first give a sufficient condition by the following Theorem~\cite{Journ2010Low}, which we extended to Hermitian matrices in a straightforward way. 
\begin{theorem}[\cite{Journ2010Low}, Theorem 7]
\label{thm:local}
  A local minimizer $\bm{Y}$ of problem~\eqref{eq:nonconvex} provides a global minimizer $\bm{X}=\bm{Y}\bm{Y}^*$ of problem~\eqref{eq:phaselift} if it is rank deficient.
\end{theorem}
\begin{proof}
  Assume the rank of $\bm{Y}\in\mathbb{C}^{n\times p}$ is $r$, we have the full rank decomposition $\bm{Y}=\tilde{\bm{Y}}\bm{M}^*$, where $\tilde{\bm{Y}}\in\mathbb{C}^{n\times r}$ and $\bm{M}\in\mathbb{C}^{p\times r}$. We take $\bm{M}_{\perp}\in\mathbb{C}^{p\times (p-r)}$ as an orthogonal basis for the orthogonal complement of the column space of $\bm{M}$. For any matrix $\tilde{\bm{Z}}\in\mathbb{C}^{n\times (p-r)}$, the matrix $\bm{Z}=\tilde{\bm{Z}}\bm{M}_{\perp}^*$ satisfies $\bm{Y}\bm{Z}^*=0$. It follows that $\trace(\bm{Z}^*\bm{S}_{\bm{Y}}\bm{Z})\geq 0$ for all matrices $\bm{Z}$. Thus the matrix $\bm{S}_{\bm{Y}}$ is semi-definite and $\bm{X}=\bm{Y}\bm{Y}^*$ is a global minimizer of problem~\eqref{eq:phaselift}.
\end{proof}

The worst case scenario is $p=n$, for which any local minimizer of problem~\eqref{eq:nonconvex} provides a global minimizer $\bm{X}=\bm{Y}\bm{Y}^*$ of problem~\eqref{eq:phaselift}.  Because, if $\bm{Y}$ is rank deficient, then the local minimizer provides a global minimizer of problem~\eqref{eq:phaselift}. Otherwise, $\bm{Y}$ is full rank, then the matrix $\bm{S}_{\bm{Y}}$ becomes zero by the equality $\bm{S}_{\bm{Y}}\bm{Y}=0$ from Lemma \ref{lem:2}. Intuitively, for PhaseLift, $p$ should not be large. In particular, the unique solution to the PhaseLift with trace regularization~\eqref{eq:phaselift} should be low rank if PhaseLift with trace minimization~\eqref{eq:sdp} is tight, e.g., for Gaussian model. In~\cite{WenHuang}, $p$ was set to below $10$ initially and then decreased further dynamically during iterations.

Here we propose to solve~\eqref{eq:nonconvex} starting from $p=p_0$ and a random initial guess $\bm{Y}_0\in\mathbb{C}^{n\times p_0}$. 
However, one may stagnate at a stationary point due to nonconvexity. The key idea is to escape from the current stationary point and decrease the objective function of~\eqref{eq:phaselift} further by increasing the column number $p$ of the decomposition factor $\bm{Y}$. 
Moreover, we construct the initialization $\bm{Y}_0\in\mathbb{C}^{n\times (p_0+1)}$ from the stationary point to problem~\eqref{eq:nonconvex} with $\bm{Y}\in\mathbb{C}^{n\times p_0}$. Hence the objective function of~\eqref{eq:phaselift} is monotonically deceasing along the $p$ increasing procedure. We successively increase $p$, until the minimum of the objective function is achieved at the global minimizer. In other words, we indirectly solve the original convex problem~\eqref{eq:phaselift} by solving a sequence of nonconvex optimization problems~\eqref{eq:nonconvex}. The remaining question is what is the criterion for one to stop increasing $p$ and the iteration. Theorem~\ref{thm:1} provides the post-verifiable condition for our purpose. 

\section{The incremental approach}
\label{sec:an-meta-algorithm}

When we solve the nonconvex problem~\eqref{eq:nonconvex} and reach a stationary point $\bm{Y}$, it may be a local minimizer or a saddle point. If  condition in Theorem~\ref{thm:1} is not satisfied, i.e., the matrix $\bm{S}_{\bm{Y}}=\nabla_{\bm{X}} f_0(\bm{Y}\bm{Y}^*)+\lambda \bm{I}$ is not positive semidefinite, it means that we have not found a global minimum of the convex problem~\eqref{eq:phaselift} yet.  The incremental algorithm is designed to monotonically decrease the objective function~\eqref{eq:phaselift} by increasing $p$.

Starting from $p=p_0$ and assume we arrive at a stationary point $\bm{Y}\in \mathbb{C}^{n\times p}$ of problem~\eqref{eq:nonconvex}. 
If $\bm{S}_{\bm{Y}}=\nabla_{\bm{X}} f_0(\bm{Y}\bm{Y}^*) + \lambda \bm{I}$ is positive semidefinite, then $\bm{Y}\bm{Y}^*$ is a solution to~\eqref{eq:phaselift} and one can terminate. Otherwise we increase $p$ to $p+1$ and minimize problem~\eqref{eq:nonconvex} by constructing an initial point $\bm{Y}_p=[\bm{Y}|\bm{0}^{n\times 1}] \in \mathbb{C}^{n\times (p+1)}$. It should be noted that if $(\nabla_{\bm{X}} f_0(\bm{Y}\bm{Y}^*)+\lambda \bm{I})\bm{Y}=0$, then $(\nabla_{\bm{X}} f_0(\bm{Y}_p\bm{Y}_p^*)+\lambda \bm{I})\bm{Y}_p=0$, i.e., $\bm{Y}_p$ is also a stationary point of problem~\eqref{eq:nonconvex} in $\mathbb{C}^{n\times (p+1)}$. 

Since $\bm{S}_{\bm{Y}}$ must have at least one negative eigenvalue, from Theorem~\ref{thm:1} and Lemma~\ref{lem:sec}, the matrix $\bm{Z}=[\bm{0}^{n\times p}|\bm{v}]\in\mathbb{C}^{n\times (p+1)}$, where $\bm{v}$ is the eigenvector responding to the smallest eigenvalue of $\bm{S}_{\bm{Y}}$, satisfies
\begin{equation*}
  \frac{1}{2}\trace(\bm{Z}^*D_{\bm{Y}}(\nabla_{\bm{X}} f_0(\bm{Y}_p\bm{Y}_p^*)+\lambda \bm{I})[\bm{Z}])=\bm{v}^T\bm{S}_{\bm{Y}_p}\bm{v}\leq 0, 
\end{equation*}
since $\bm{Y}_p\bm{Z}^*=0$. Thus $\bm{Z}$ is a descent direction of~\eqref{eq:nonconvex} at $\bm{Y}_p$ in $\mathbb{C}^{n\times (p+1)}$. This allows us to escape the saddle point $\bm{Y}_p$ along the descent direction and minimize the objective function further through one backtracking step. After one decreasing step, we reach a new point $\bm{Y}_0\in\mathbb{C}^{n\times(p+1)}$ and it provides a good initial guess for the next optimization problem~\eqref{eq:nonconvex} for $\bm{Y}\in \mathbb{C}^{n\times (p+1)}$. This procedure is done repeatedly till the termination of the algorithm. Note that simply padding a zero column to current stationary point provides a good initial guess for the next stage. 
We summarize all the above elements in Algorithm~\ref{alg:inc}. We call the incremental nonconvex optimization algorithm to solve PhaseLift~\eqref{eq:phaselift} for phase retrieval \texttt{IncrePR}. The parameter $\epsilon$ is a threshold on the eigenvalues of $\bm{S}_{\bm{Y}}$ to decide the nonnegativity of this matrix. $\epsilon$ is chosen depending on the problem and desired accuracy.

\begin{algorithm}[H]
  \caption{Incremental algorithm (\texttt{IncrePR}) for solving problem~\eqref{eq:phaselift}\label{alg:inc}}
  \begin{algorithmic}[1]
    \REQUIRE Initial column number $p_0$, initial guess $\bm{Y}_0\in\mathbb{C}^{n\times p_0}$ and tolerance parameter $\epsilon$.
\ENSURE Solution $\bm{X}=\bm{Y}_p\bm{Y}_p^*$ to problem~\eqref{eq:phaselift}.
\STATE Initialize $p=p_0$
\REPEAT
\STATE Apply an optimization method to find a stationary point $\bm{Y}_p$ of problem~\eqref{eq:nonconvex} from initial guess $\bm{Y}_0$.
\STATE Find the smallest eigenvalue $\nu_{\min}$ of $\bm{S}_{\bm{Y}_p}$ defined in~\eqref{eq:2}.
\IF {$\nu_{\min}\geq -\epsilon$}
\STATE \textbf{break}
\ELSE
\STATE Find the corresponding eigenvector $\bm{v}$ of the matrix $\bm{S}_{\bm{Y}_p}$.
\STATE $p\leftarrow p+1$, $\bm{Y}_p\leftarrow [\bm{Y}_p|\bm{0}]$
\STATE Apply one descent step along descent direction $\bm{Z}_p=[\bm{0}|\bm{v}]$ (such as backtracking line-search), arrive at point $\bm{Y}_0$.
\ENDIF
\UNTIL{\textbf{convergence}}
  \end{algorithmic}
\end{algorithm}

Based on Theorem~\ref{thm:1}, positive semidefiniteness of the matrix $\bm{S}_{\bm{Y}}$ at a stationary point of the nonconvex problem~\eqref{eq:nonconvex} indicates global convergence of \texttt{IncrePR}. Since all eigenvalues of Hermitian matrix are real-valued, we just have to compute the smallest eigenvalue of $\bm{S}_{\bm{Y}}$. In particular,  we use the LOBPCG routine to locate the smallest eigenvalue without the need to form the matrix $\bm{S}_{\bm{Y}}$~\cite{Knyazev2006Toward}.

In general, one can employ two explicit criteria to terminate our incremental algorithm in terms of accuracy and the computational cost respectively. The first criteria is to check if the matrix $\bm{S}_{\bm{Y}}$ is positive semidefinite at a stationary point of the nonconvex problem~\eqref{eq:nonconvex} within the threshold $\epsilon$. The smaller the absolute value of the smallest negative eigenvalue of matrix $\bm{S}_{\bm{Y}}$, the closer the computed solution is to a global minimizer of~\eqref{eq:phaselift}.
In other words, absolute value of the smallest eigenvalue of matrix $\bm{S}_{\bm{Y}}$ can be used as an indicator of the progress of the incremental algorithm and can be used to monitor the convergence. The upper limit of $p$ for $\bm{Y}$ can be used as another criteria to forcibly limit the computation and storage cost to solve the nonconvex problem~\eqref{eq:nonconvex}. It is worth to note that our incremental approach always decreases the objective function~\eqref{eq:phaselift} no matter what the error tolerance $\epsilon$ is for each nonconvex problem. 
In particular, one does not need to set it too small at the beginning and then decrease it during the incremental process.

Compared to the dynamical rank decreasing strategy used in~\cite{WenHuang}, our incremental strategy is more effective and natural. In our algorithm, one only needs to solve a sequence of unconstrained nonconvex problems~\eqref{eq:nonconvex}. 
For each nonconvex problem, one simply use a first-order method to find a stationary point and use Theorem~\ref{thm:1} to decide whether to continue or terminate. However, if only a stationary point is located for each nonconvex problem,  Algorithm~\ref{alg:inc} may not find the solution and terminate even when $p$ reaches $n$.
In this worst case scenario,  one should compute a local minimizer of \eqref{eq:nonconvex} by second-order methods and apply Theorem~\ref{thm:local}. Compared to first-order method, to find a local minimizer of \eqref{eq:nonconvex}, second-order optimality condition has to be checked when a stationary point is reached and, if the stationary point is a saddle point, a negative curvature direction needs to be found. For the application of phase retrieval, the worst case scenario happen very rarely. Thus we typically use a first order method to locate a stationary point in our algorithm. Although, applying first-order method may lead to a larger termination $p$, numerical experiments show that both first-order and second-order methods for \eqref{eq:nonconvex}  work well and terminate with small $p$  for phase retrieval applications. In the next section, we give the explicit gradient and Hessian for two data fidelity functions~\eqref{eq:func} and~\eqref{eq:pfunc} and briefly review optimization methods for completeness. 

\section{Unconstrained optimization}
\label{sec:riem-optim}

In this section, we provide components of the optimization scheme for solving the following nonconvex unconstrained problem:
\begin{equation*}
  \min_{\bm{Y}\in\mathbb{C}^{n\times p}}\quad \tilde{f}(\bm{Y}) = f_0(\bm{Y}\bm{Y}^*).
\end{equation*}
For the trace regularized minimization \eqref{eq:nonconvex} ($\lambda\ne 0$), the gradient and Hessian have an additional simple term. For simplicity we skip the trace term in the following derivation.

For the nonconvex objective $\tilde{f}_0$ of~\eqref{eq:func}, its gradient and Hessian along the direction $\bm{\xi}_{\bm{Y}}$ are given by
\begin{equation*}
\begin{aligned}
  \nabla \tilde{f}_0 (\bm{Y}) &=\frac{1}{m}\mathcal{A}^*(\mathcal{A}(\bm{Y}\bm{Y}^*)-\bm{b})\bm{Y},\\
  \nabla^2 \tilde{f}_0(\bm{Y})[\bm{\xi}_{\bm{Y}}] & = \frac{1}{m}\mathcal{A}^*(\mathcal{A}(\bm{Y}\bm{Y}^*)-\bm{b})\bm{\xi}_{\bm{Y}} + \frac{1}{m}\mathcal{A}^*(\mathcal{A}(\bm{Y}\bm{\xi}_{\bm{Y}}^*+\bm{\xi}_{\bm{Y}}\bm{Y}^*))\bm{Y}.
\end{aligned}
\end{equation*}

For the nonconvex objective $\tilde{f}_0$ of~\eqref{eq:pfunc}, its gradient and Hessian along the direction $\bm{\xi}_{\bm{Y}}$ are given by
\begin{equation*}
\begin{aligned}
  \nabla \tilde{f}_0 (\bm{Y}) &= 2\mathcal{A}^*\left(\bm{1}-\frac{\bm{b}}{\mathcal{A}(\bm{Y}\bm{Y}^*)}\right)\bm{Y},\\
  \nabla^2 \tilde{f}_0(\bm{Y})[\bm{\xi}_{\bm{Y}}] & = 2\mathcal{A}^*\left(\bm{1}-\frac{\bm{b}}{\mathcal{A}(\bm{Y}\bm{Y}^*)}\right)\bm{\xi}_{\bm{Y}} + 2\mathcal{A}^*\left(\frac{\bm{b}}{(\mathcal{A}(\bm{Y}\bm{Y}^*))^2}\circ\left(\mathcal{A}(\bm{Y}\bm{\xi}_{\bm{Y}}^*+\bm{\xi}_{\bm{Y}}\bm{Y}^*)\right)\right)\bm{Y}.
\end{aligned}
\end{equation*}
The involved computation $\mathcal{A}(\bm{Y}\bm{Y}^*)$ and $\mathcal{A}^*(\cdot)$ is done by matrix-vector multiplication. When dealing with structured measurements,  such as Fourier phase retrieval and the associated coded diffraction pattern~\cite{Candes2014}, the computation can be done more efficiently.

Note that the function $\tilde{f}(\bm{Y})$ is invariant under an orthonormal transformation. Thus the stationary points are not isolated, which can cause complications for second-order optimization methods~\cite{Absil2009Optimization}. 
To locate a local minimizer for a nonconvex optimization, a second-order method, such as trust-region method, is needed. Since our nonconvex optimization is invariant up to an orthonormal transformation, one should consider the Riemannian optimization within the quotient space, e.g., the Manopt package~\cite{Boumal2013Manopt}, which only requires inputs of the gradient and Hessian in Euclidean space.

On the other hand, the invariance does not affect first-order methods. To avoid computing projections onto the quotient space at each iteration step, where Sylvester equation and Newton equation for Riemannian optimization need to be solved~\cite{Boumal2013Manopt}, one can use first-order line search method to locate a stationary point for each nonconvex optimization \eqref{eq:nonconvex}. First-order line search method constructs a sequence,
\begin{equation}
  \label{eq:7}
  \bm{Y}_{k+1} = \bm{Y}_k+\alpha_k\bm{d_k},
\end{equation}
where
$\bm{d}_k\in\mathbb{C}^{n\times p}$ is a decent direction and 
$\alpha_k\in\mathbb{R}$ is the step size along the descent direction at $k$th iteration. 
The simplest choice, a.k.a.
gradient descent, of $\bm{d_k}$ is $- \nabla \tilde{f}(\bm{Y}_k)$. Although first-order method for \eqref{eq:nonconvex} with fixed $p$ only converges to a stationary point, it is suffice since the incremental approach decreases the function~\eqref{eq:phaselift} by increasing $p$ to $p+1$. Other more efficient methods with different choices of descent directions, such as CG and LBFGS~\cite{Li161,Li16} methods, can also be used. 

There is no explicit line search solution for step size $\alpha_k$ for the likelihood estimation function~\eqref{eq:pfunc}. The inexact line search step size should satisfy the well known Wolfe conditions (whereas $\bm{g}$ denotes $\nabla \tilde{f}$):
\begin{subequations}
  \begin{equation}
    \label{eq:10}
    \tilde{f}(\bm{Y}_k + \alpha_k \bm{d}_k)\leq \tilde{f}(\bm{Y}_k) + c_1\alpha_k \re(\bm{d}_k^*\bm{g}_k)
  \end{equation}
and
\begin{equation}
  \label{eq:11}
  \re (\bm{d}_k^*\bm{g}_{k+1}) \geq c_2 \re(\bm{d}_k^*\bm{g}_k),
\end{equation}
\end{subequations}
where condition $0<c_1<c_2<1$ is satisfied. Equations~\eqref{eq:10} and~\eqref{eq:11} are
known as the sufficient decrease and curvature condition
respectively~\cite{Nocedal2006}. For intensity-based least-squares function \eqref{eq:func}, one can find the step size $\alpha_k$ from exact line search by solving a cubic equation. A similar routine can be found in~\cite{Li16} to solve  the nonconvex optimization~\eqref{eq:nonconvex} with $p=1$. 

\section{Solving phase retrieval with restart}

For noiseless Gaussian model, when the number of measurements $m=\mathcal{O}(n\log n)$, the parameter $\lambda$ in~\eqref{eq:phaselift} can be set to zero, one can solve~\eqref{eq:phaselift} by \texttt{IncrePR} with $p_0=1$. If $m$ is between $\mathcal{O}(n)$ and $\mathcal{O}(n\log n)$, PhaseLift with trace minimization~\eqref{eq:sdp} is tight. We can solve a series of problems~\eqref{eq:phaselift} with different $\lambda$'s to approach the solution to~\eqref{eq:sdp}. Standard continuity approach of solving~\eqref{eq:phaselift} with a series of decreasing $\lambda$'s can be exploited, where solution to current $\lambda$ is the initial guess for next $\lambda$. The continuity approach is used to select the optimal regularization parameter for low rank matrix recovery problem~\cite{Mishra2013Low}.

For Gaussian model, if the number of measurements $m=\mathcal{O}(n)$, the effect of trace regularization in~\eqref{eq:sdp} becomes important.
We propose the following three step algorithm with a restart strategy. We start with solving~\eqref{eq:phaselift} with $\lambda=0$ from a random initial guess. However, 
the solution $\bm{X}_1$ is only aiming to fit the data and not of rank one in general, since the solution to the convex problem~\eqref{eq:phaselift} with $\lambda=0$ may not be unique from $\mathcal{O}(n)$ measurements. Then~\eqref{eq:phaselift} with $\lambda=\lambda_0$ is solved starting from $\bm{X}_1$ where trace regularization is used to promote low rank. The solution $\bm{X}_2$ should be close to the solution to problem~\eqref{eq:sdp}. We then solve~\eqref{eq:phaselift} with $\lambda=0$ again starting from a rank one approximation of $\bm{X}_2$. Solving the first two problems can be regarded as the initialization stage and solving the last problem is the refinement stage. We use  \texttt{IncrePR} (Algorithm~\ref{alg:inc}) to solve each SDP. Since a rank one approximation is extracted and used to feed into the last step, we call it a restart strategy.
The detailed description of the three-step algorithm for synthetic Gaussian model is presented in Algorithm~\ref{alg:2}.  
In the case of $m=\mathcal{O}(n\log n)$ measurements for noiseless Gaussian model, Algorithm~\ref{alg:2} terminates at stage one, since solution $\bm{X}_1$ is exactly of rank one~\cite{Demanet2013}. For noisy measurements, we also use  Algorithm~\ref{alg:2} to solve~\eqref{eq:phaselift} to obtain a stable solution.

\begin{algorithm}[H]
  \caption{\texttt{IncrePR} with restart for Gaussian model\label{alg:2}}
  \begin{algorithmic}[1]
    \REQUIRE Initial $p_0=1$, penalty parameter $\lambda_0$ and initialization $\bm{Y}_0\in\mathbb{C}^{n}$.
\ENSURE An approximate solution $\bm{x}$ to problem~\eqref{eq:1}.
\STATE Initialize $p=p_0$ and $\lambda=0$, use \texttt{IncrePR} to solve~\eqref{eq:phaselift} and obtain solution $\bm{X}_1=\bm{Y}_1\bm{Y}_1^*$.
\IF {$\rank(\bm{Y}_1)=1$}
\STATE Set $\bm{x}=\bm{Y}_1$
\STATE \textbf{break}
\ELSE 
\STATE Set $p_0$ to the column number of $\bm{Y}_1$ and $\lambda=\lambda_0$. Initialize $\bm{Y}_0=\bm{Y}_1$.
\STATE Use \texttt{IncrePR} to solve~\eqref{eq:phaselift} and obtain solution $\bm{X}_2=\bm{Y}_2\bm{Y}_2^*$.
\STATE Extract the best rank one approximation $\bm{y}_2$ to $\bm{Y}_2$.
\STATE Set $p_0=1$, $\lambda=0$, and initialize $\bm{Y}_0=\bm{y}_2$.
\STATE  Use \texttt{IncrePR} to solve~\eqref{eq:phaselift} and obtain solution $\bm{X}_3=\bm{Y}_3\bm{Y}_3^*$. 
\STATE Extract the rank one approximation $\bm{x}$ of $\bm{Y}_3$.
\ENDIF 
  \end{algorithmic}
\end{algorithm}

We pause here to compare \texttt{IncrePR} with the PhaseLift solver in~\cite{Candes2013} . First of all, the two algorithms solve the same SDR PhaseLift. So if the solution to SDR PhaseLift is unique, then PhaseLift and \texttt{IncrePR} give the same solution. The difference is that \texttt{IncrePR} uses an incremental nonconvex solver instead of the typical SDP solver. SDP solver operates on an $n\times n$ matrix, while \texttt{IncrePR} operates on an $n\times p$ matrix, and $p$ is generally less than ten, as demonstrated in numerical tests. At each iteration, the computation complexity of the eigenvalue decomposition of SDP solver is $\mathcal{O}(n^3)$ for projection onto $\mathbb{H}_+^n$.  For \texttt{IncrePR}, the main cost is checking the positive semi-definiteness of $\bm{S}_{\bm{Y}_p}$ defined in~\eqref{eq:2} and finding its smallest eigenvalue and the corresponding eigenvector. This operation has to be performed each time when $p$ is increased. 
\texttt{IncrePR}  with restart can also achieve sharper phase transition $\mathcal{O}(1)$. $\lambda$ is set to $10^{-3}$ in~\cite{Candes2013} and hence the penalty of trace term generally is not strong enough to find a good approximate solution from limited number of measurements. Thus its best phase transition is roughly $\mathcal{O}(\log n)$~\cite{Demanet2013}.

For other types of more practical measurements, such as the transmission matrix and Fourier measurements, the tightness of~\eqref{eq:sdp} is not guaranteed. However, one can still use \texttt{IncrePR} to obtain the global minimizer of SDR and use its rank one approximation
as a good approximation to the global minimizer of the nonconvex problem~\eqref{eq:nonconvex} with $p=1$. In particular, one can optimize~\eqref{eq:phaselift} using \texttt{IncrePR} with the restart strategy (Algorithm~\ref{alg:2}) a few times, where the current solution is the initialization of next round of Algorithm~\ref{alg:2}. If needed, the rank one approximation obtained from \texttt{IncrePR} can be used as an initial guess for a nonconvex solver for the nonconvex problem~\eqref{eq:nonconvex} with $\lambda=0$.
 
\section{Numerical tests}
\label{sec:numerical-test}

\subsection{Gaussian model}

For Gaussian model, a benchmark for phase retrieval~\eqref{eq:1} is to compare the phase transition of each algorithm. The \emph{phase transition} is the critical point of the ratio of the number of measurements $m$ to the length of the signal $n$. When $m/n$ is above the phase transition, the unique solution to~\eqref{eq:1} can always be located by the algorithm. 

We would like to illustrate the empirical phase transition of \texttt{IncrePR} with restart for Gaussian model. This is done by Monte-Carlo simulations. Here we fix the length of signal $n=400$ and vary the number of measurements $m=\{1n,1.1n,\cdots, 5n\}$. For each pair $(n,m)$, we generate $50$ instances of problem~\eqref{eq:1}. For each instance, we generate a 1-dimension real/complex signal of length $n=400$ at random, and the sampling vectors $\bm{a}_r$'s drawn from real/complex Gaussian distribution and obtain the measurement intensity data. We then compute the \emph{recovery rate}, i.e., the percentage of the successful recoveries for the $50$ instances. We use the relative error to measure the success of recovery. We denote the solution computed by  \texttt{IncrePR} with restart as $\bm{x}$. The relative error is defined
as 
\begin{equation}
  \label{eq:18}
  relerr = \min_{\lvert c\rvert=1}\frac{\norm{c\bm{x}-\bm{x}^{\natural}}_2}{\norm{\bm{x}^{\natural}}_2},
\end{equation}
where $c$ is to get rid of the effect of the constant phase shift for phase retrieval problem.

We consider a signal to be perfectly recovered if the relative
error~\eqref{eq:18} is below $10^{-5}$. The least-squares data fidelity~\eqref{eq:func} is used in problem~\eqref{eq:phaselift}. For real and complex cases, $\lambda_0$ in Algorithm~\ref{alg:2} is set to $10^2$. The maximum iteration for each nonconvex problem is set to $1000$ and $3000$ for real and complex models respectively. For real case, the termination $\epsilon$ for incremental nonconvex approach is set to $100$, $500$ and $1$ for the three convex problems respectively. For complex case, the termination $\epsilon$ is set to $800$, $1000$ and $1$ respectively. We do not need to set $\epsilon$ too small for the first two stages, since they mainly serve to provide a good initial guess for the last stage through the restart strategy. Before exploring the phase transition, we first study the effect of first- and second-order methods on the termination column number $p$ for \texttt{IncrePR}.

\paragraph{Termination column number} 

We run 10 Monte Carlo trials and compare the performance of the conjugate gradient (CG) method with that of the second-order trust-region (TR) method. In Table~\ref{tab:1}, we list the termination column number $p$ of \texttt{IncrePR} with restart (the three-stage Algorithm~\ref{alg:2}) at each stage, where the average, minimum and maximum termination column number $p$ (in parenthesis separated by commas) are shown. For real case, we test three groups with the number of measurements $m=2n,2.5n,3n$ respectively. For complex case, the number of measurements $m=3n,3.5n,4n$ are considered. The termination $p$ for CG is typically larger than that of TR, since the first-order method only locates a stationary point, while TR finds a local minimizer, for the nonconvex problem \eqref{eq:nonconvex}. For the challenging case of minimal measurement limit, $m=2n$ for real case and $m=3n$ for complex case, there are tests for which the termination $p>1$ at the third stage. It shows the necessity of utilizing \texttt{IncrePR} at the refinement stage to solve the convex problem. If nonconvex formulation~\eqref{eq:nonconvex} with fixed $p=1$ is used to refine the solution, it may still stagnate at a stationary point. 
\begin{table}
\setlength{\tabcolsep}{.5\tabcolsep}
\centering
\caption{Comparison of termination column number $p$ of \texttt{IncrePR} with restart (Algorithm~\ref{alg:2}) by first-/second-order method (CG/TR). The average, minimum and maximum termination $p$'s over 10 runs are shown for real and complex Gaussian models.\label{tab:1}}
\begin{minipage}[b]{.48\textwidth}
  \scalebox{0.7}{
\begin{tabular}{l@{}c@{}cc@{}c@{}cc@{}c@{}cc}
\toprule
Real &\phantom{ab}& 
\multicolumn{2}{c}{$m=2n$} & 
\phantom{ab} & 
\multicolumn{2}{c}{$m=2.5n$} &
\phantom{ab} & 
\multicolumn{2}{c}{$m=3n$}\\
  \cmidrule{3-4} \cmidrule{6-7} \cmidrule{9-10}
stage && CG & TR && CG & TR &&CG & TR\\
\midrule
I && 4.6 (4, 5) & 2.6 (2, 3) &&3.9 (3, 5) &2.2 (2, 3) &&3.3 (2, 4) &2.0 (2, 2)\\
II && 6.3 (5, 8) &2.9 (2, 3) && 3.9 (3, 5) &2.2 (2, 3) &&3.3 (2, 4) &2.0 (2, 2) \\
III && 1.1 (1, 2) &1.0 (1, 1) &&1.0 (1, 1) &1.0 (1, 1) &&1.0 (1, 1) &1.0 (1, 1)\\
\bottomrule
\end{tabular}}
\end{minipage}
\begin{minipage}[b]{.48\textwidth}
  \scalebox{0.7}{
\begin{tabular}{l@{}c@{}cc@{}c@{}cc@{}c@{}cc}
\toprule
Compl &\phantom{ab}& 
\multicolumn{2}{c}{$m=3n$} & 
\phantom{ab} & 
\multicolumn{2}{c}{$m=3.5n$} &
\phantom{ab} & 
\multicolumn{2}{c}{$m=4n$}\\
  \cmidrule{3-4} \cmidrule{6-7} \cmidrule{9-10}
stage && CG & TR && CG & TR &&CG & TR\\
\midrule
I && 2.6 (2, 3) &2.0 (2, 2) &&1.6 (1, 2) &1.8 (1, 2) &&1.5 (1, 2) &1.4 (1, 2) \\
II && 2.7 (2, 4) &2.0 (2, 2) &&1.6 (1, 2) &1.8 (1, 2) &&1.5 (1, 2) &1.4 (1, 2) \\
III && 1.4 (1, 2) &1.2 (1, 2) &&1.0 (1, 1) &1.0 (1, 1) &&1.0 (1, 1) &1.0 (1, 1) \\
\bottomrule
\end{tabular}}
\end{minipage}
\end{table}

\paragraph{Phase transition} In the following experiments, we study the phase transition of \texttt{IncrePR} with restart with comparison to other algorithms. We still use the least-squares data-fidelity~\eqref{eq:func} and apply the conjugate gradient method to solve each nonconvex optimization. We compare \texttt{IncrePR} (Algorithm~\ref{alg:2}) using random initial guess with state-of-the-art nonconvex methods, including reweighted amplitude flow (RAF)~\cite{Wang2017Solving}, truncated amplitude flow (TAF)~\cite{Wang16}, truncated Wirtinger flow (TWF)~\cite{Chen2015b} as well as Wirtinger flow (WF)~\cite{Candes2015} with the same random initial guess as \texttt{IncrePR} or a good initial guess provided by the state-of-the-art reweighted maximal correlation method~\cite{Wang2017Solving}. 
Here we only compare with nonconvex algorithms, since \texttt{IncrePR} and PhaseLift should share the same phase transition if the restart strategy is also utilized in PhaseLift. For fair comparison, we set the parameters involved in each method to their default values recommended in the literature. The maximum iteration number of the nonconvex solvers is set to $3000$ and $6000$ for real and complex cases respectively. 

For each pair $(n,m)$, we compute the recovery rate of each method by averaging 50 Monte Carlo trials. Phase transition of each method is depicted in Figure~\ref{fig:1}. \texttt{IncrePR} with restart shows the sharpest phase transition for both real and complex cases. It achieves perfect recovery from $m\geq 1.9n$ and $m\geq 2.7n$ measurement data for the real and complex cases respectively. It beats the theoretical guarantee of $m=2n-1$~\cite{Balan2006} and $m\geq 4n-4$~\cite{Balan2015} for generic real and complex cases respectively. The numerical results should not be alarming since the theoretical result is for generic measurement vectors, which include Gaussian model as a specific case. As stated in~\cite{Wang2017Solving}, the reweighted amplitude flow (\texttt{RAF}) can achieve perfect recovery for signal with dimension $n\geq 2000$ from $m=2n$ measurement data, while it may encounter difficulty for small dimension signal ($n=400$ here), as shown in Figure~\ref{fig:1}. It results from the poor initialization by the reweighted maximal correlation method, which requires large signal dimension (e.g., $n\geq 2000$). 
It is also obvious that the performance of nonconvex methods degrades significantly when starting from a random initialization. \texttt{IncrePR} with restart outperforms all the state-of-the-art nonconvex solvers markedly from random initialization. This is due to the advantage of convex PhaseLift model.
 
\begin{figure}
     \begin{subfigure}[b]{.45\textwidth}
     \centering
     \includegraphics[width=\textwidth]{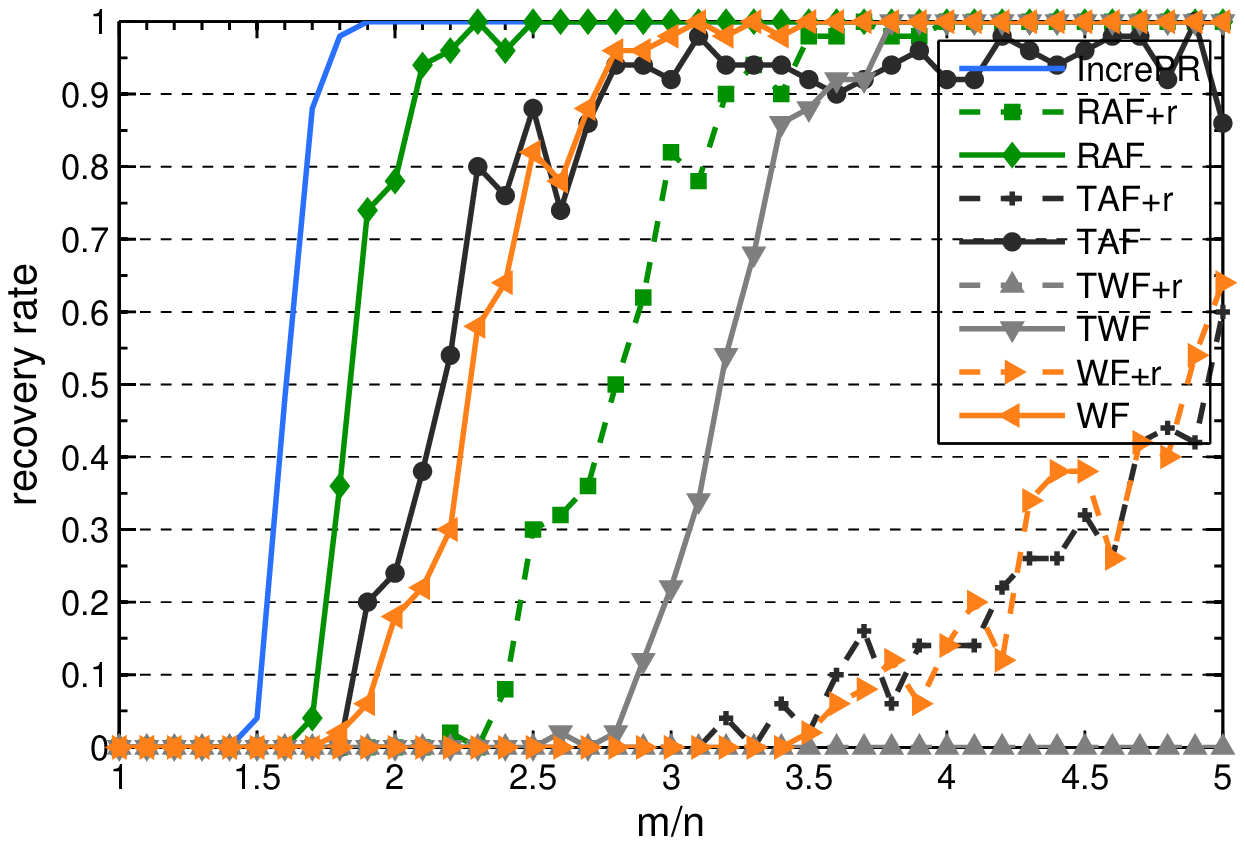}
     \caption{Standard real Gaussian model\label{fig:1a}}
   \end{subfigure}
 \begin{subfigure}[b]{.45\textwidth}
     \centering
     \includegraphics[width=\textwidth]{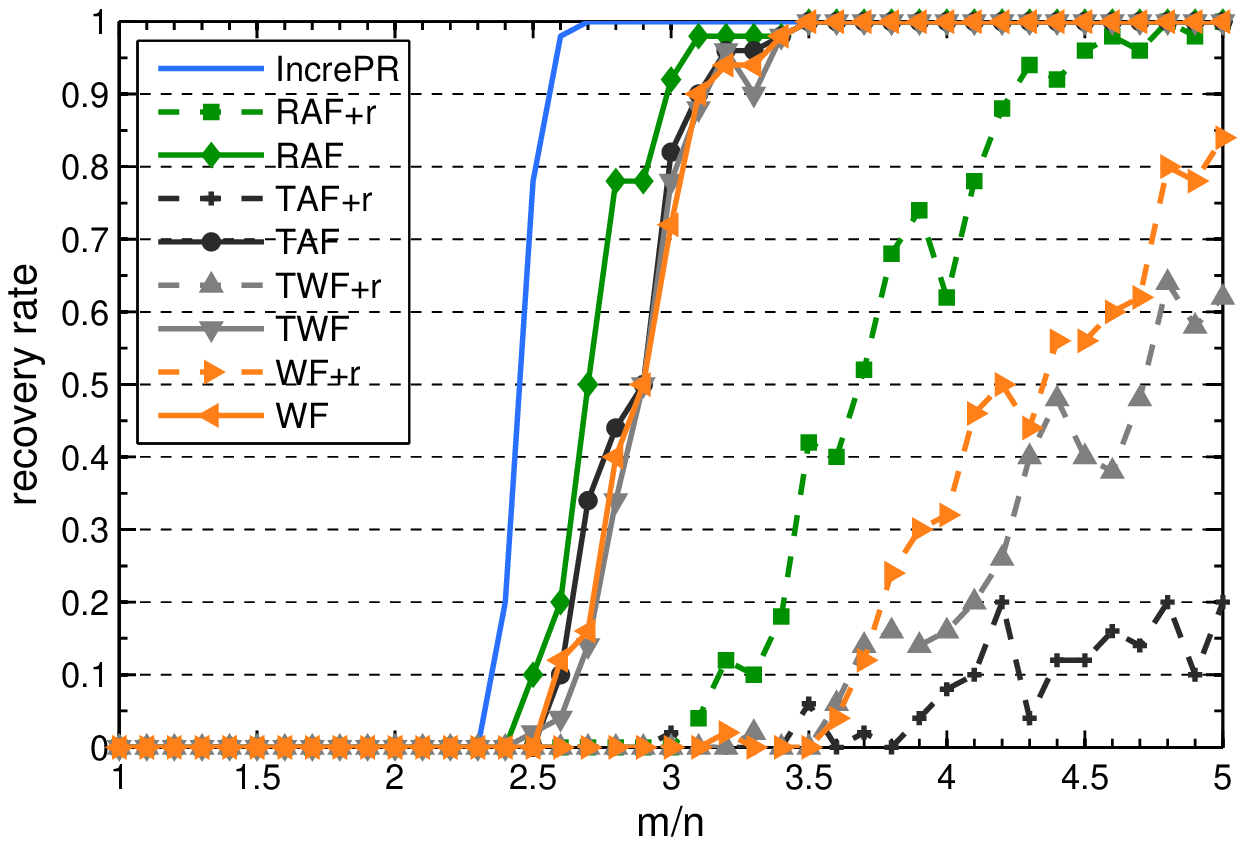}
     \caption{Standard complex Gaussian model\label{fig:1b}}
   \end{subfigure}
  \caption{Phase transition for noiseless Gaussian model. The legend with letter 'r' stands for starting from the same random initial guess as \texttt{IncrePR}.\label{fig:1}}
\end{figure}

\paragraph{Effect of trace penalty}

As mentioned before, the trace minimization in~\eqref{eq:sdp} is not necessary provided the number of measurements $m$ is on the order of $n\log n$. In this case, \texttt{IncrePR} with restart terminates at stage 1 (step 4) in Algorithm~\ref{alg:2}, i.e, solution $\bm{X}_1$ is of rank one. To demonstrate the effect of trace penalty in~\eqref{eq:phaselift} when $m$ is inbetween, we extract the rank one approximations $\bm{x}_1$ and $\bm{x}_3$ to the solutions $\bm{Y}_1$ and $\bm{Y}_3$ respectively. If $\bm{Y}_1$ is of rank one, we take $\bm{Y}_3=\bm{Y}_1$. We trace $\bm{x}_1$ after the first stage, we denote it \texttt{IncrePR-1}. We calculate the relative error between $\bm{x}_1$ ($\bm{x}_3$) and $\bm{x}^{\natural}$ to judge whether perfect recovery is obtained or not. Figure~\ref{fig:2} depicts the effect of trace regularization. The phase transition curves and the average relative errors of \texttt{IncrePR-1} and \texttt{IncrePR} are illustrated in Figure~\ref{fig:2a} and~\ref{fig:2c} respectively. The phase transition points of \texttt{IncrePR} and \texttt{IncrePR-1} are $m=1.9n$ ($m=2.7n$) and $m=2.4n$ ($m=3.6n$) for real (complex) Gaussian models respectively. 
\begin{figure}
\centering
     \begin{subfigure}[t]{.45\textwidth}
     \centering
     \includegraphics[width=\textwidth]{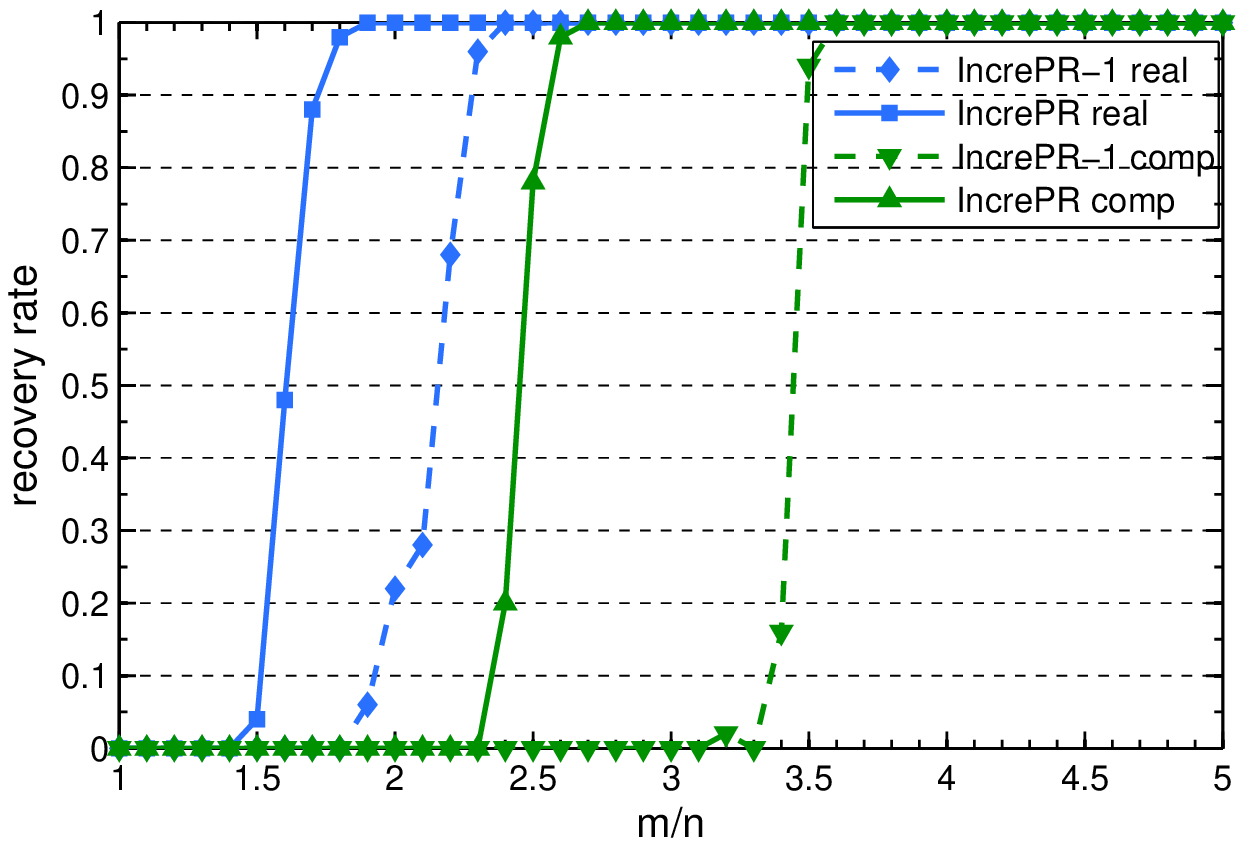}
     \caption{Recovery rate for Gaussian model\label{fig:2a}}
   \end{subfigure}
       \begin{subfigure}[t]{.45\textwidth}
     \centering
     \includegraphics[width=\textwidth]{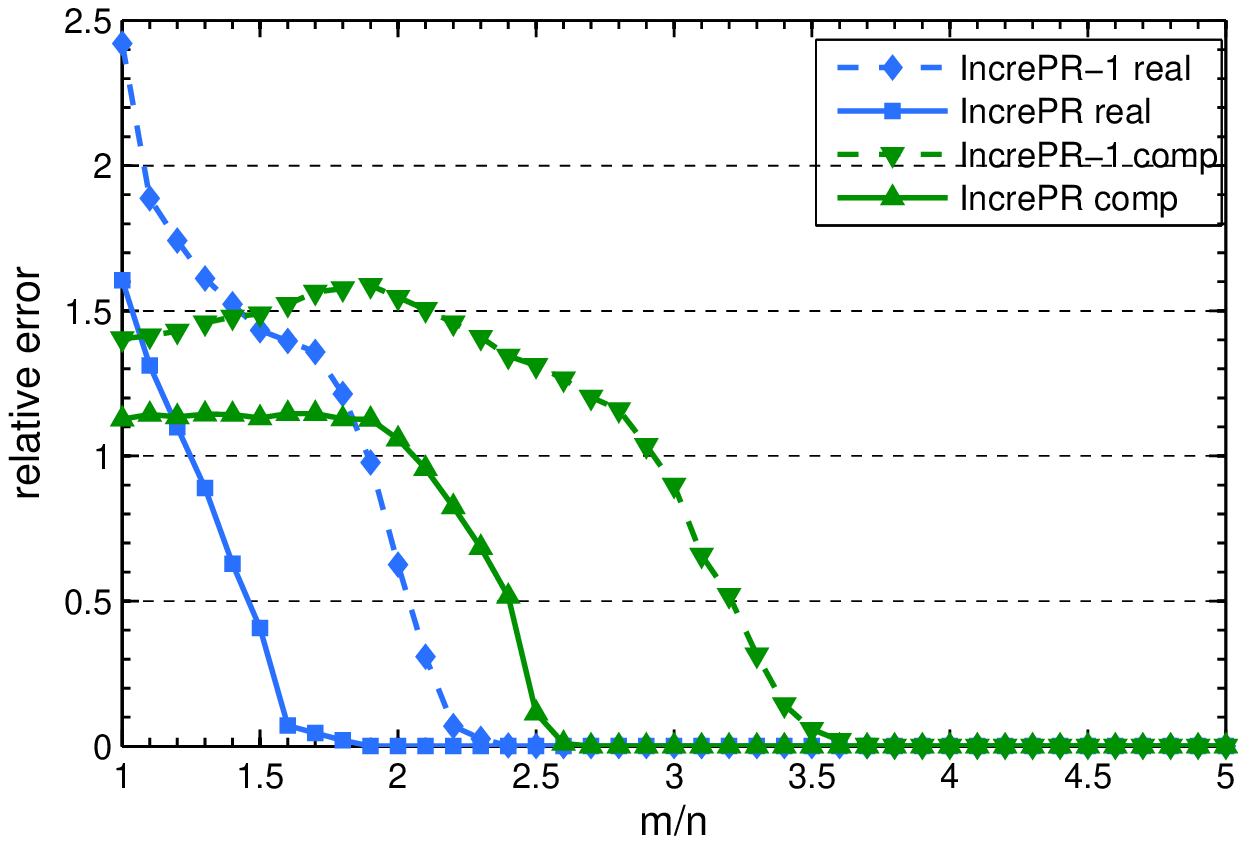}
     \caption{Average relative error for Gaussian model\label{fig:2c}}
   \end{subfigure}
\caption{The effect of trace penalty. \texttt{IncrePR-1} stands for rank one approximation to the solution from the first-stage of \texttt{IncrePR} (Algorithm~\ref{alg:2}) .\label{fig:2}}
\end{figure}

\paragraph{Noisy measurements}
We demonstrate the stability of \texttt{IncrePR} with restart for noisy Gaussian model. We only consider additive Gaussian noise for the real case. We vary the noise
level from $10$dB to $50$dB with an increment of $5$dB. We only compare \texttt{IncrePR} with the \texttt{RAF} method for its superior performance among the state-of-the-art nonconvex solvers. Figure~\ref{fig:4}
illustrates the stability of \texttt{IncrePR} with restart algorithm. For all tests, \texttt{IncrePR} with restart shows better performance, compared to \texttt{RAF}. The improved recoveries for $m=2n$ and $m=3n$ are in particular remarkable.  Due to the stability of convex PhaseLift formulation~\eqref{eq:sdp}, the capability of \texttt{IncrePR} in dealing with noisy data from minimal measurements is very promising. The poor performance for \texttt{RAF} is due to its poor initialization from noisy and limited measurement data. 
\begin{figure}
   \centering
     \includegraphics[width=.5\textwidth]{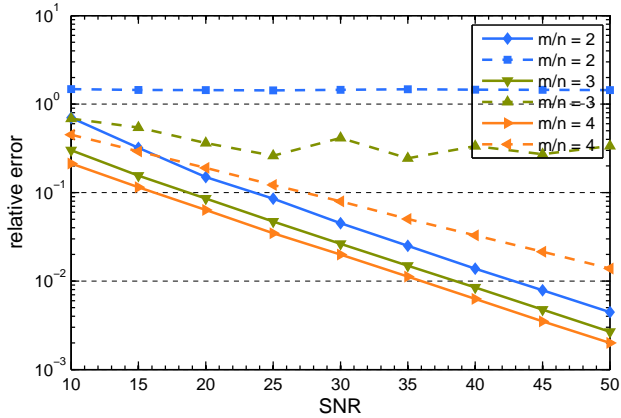}
  \caption{Average relative error for 50 trials of real noisy Gaussian model. RAF and  \texttt{IncrePR} are represented by dotted and solid line respectively.\label{fig:4}}
\end{figure}

\subsection{Real transmission measurements}

Now we test on a real example for which the measurement transmission matrix is provided by Phasepack~\cite{Chandra2017PhasePack}. The package is used to benchmark the performance of various phase retrieval algorithms. The transmission matrix dataset provides measurement matrices at three different image resolutions ($16\times 16$, $40\times 40$ and $64\times 64$). The rows of the transmission matrices are calculated using a measurement process (also a phase retrieval problem), and some are more accurate than the others. Each measurement matrix comes with a per-row residual to measure the accuracy of that row. For our test, we use the measurement matrix \verb+A_prVAMP.mat+ of image with resolution $16\times 16$. We can cut off the measurement matrix to a smaller size by only loading the most accurate rows. We obtain the measurement matrix by setting the residual constant to $0.04$. The signal to recover is a 1-dimensional signal with length $n=16\times 16=256$. 

Since we are considering a non-Gaussian model, the PhaseLift model is not tight for this example. We would not expect that the rank one approximation obtained from the PhaseLift model is the true solution except that it may be close to the global minimizer of the nonconvex problem~\eqref{eq:nonconvex} with $p=1$. We directly adopt Algorithm~\ref{alg:2} and extract the rank one approximation of the solution. At the end, we refine the rank one approximation by nonconvex optimization of~\eqref{eq:nonconvex} with $\lambda=0$ by \texttt{RAF}.  Our numerical experiments show that the local minimizer found by \texttt{IncrePR} plus refinement from a random initialization (Alg.2 + r) is better than that found by directly solving the nonconvex problem~\eqref{eq:nonconvex} with $p=1$ by \texttt{RAF} from random initialization (\texttt{RAF+r}) or from the reweighted maximal correlation initialization (\texttt{RAF}). 

\begin{table}
\setlength{\tabcolsep}{.5\tabcolsep}
\centering
\caption{The average and minimum/maximum relative errors of reconstructions over 10 runs are shown for real transmission measurements.\label{tab:2}}
\scalebox{.9}{\begin{tabular}{l@{}c@{}c@{}c@{}c@{}c@{}c@{}c@{}c}
\toprule
 &\phantom{ab}& 
$m=2n$& 
\phantom{ab} & 
$m=3n$ &
\phantom{ab} & 
$m=4n$&
\phantom{ab} & 
$m=5n$\\
\midrule
\texttt{Alg.2+r} &&0.5407 (0.5107, 0.5635) &&0.5156 (0.5054, 0.5283) &&0.5113 (0.5079, 0.5179) &&0.5008 (0.4958, 0.5055) \\
\texttt{RAF} &&0.7532  &&0.6770  &&0.5904  &&0.5524 \\
\texttt{RAF+r} &&0.8250 (0.8222, 0.8276) &&0.8251 (0.8189, 0.8272) &&0.8154 (0.7724, 0.8273) &&0.7885 (0.5491, 0.8242) \\
\bottomrule
\end{tabular}}
\end{table}

We compare Alg.2+r with \texttt{RAF+r} and \texttt{RAF} for the transmission dataset for the number of measurements $m=2n,3n,4n,5n$. For \texttt{IncrePR}, the least-squares data fidelity~\eqref{eq:func} is used and the termination $\epsilon$ is set to $100$. We run $10$ experiments for Alg.2 + r and \texttt{RAF+r} with the same random initialization. For \texttt{RAF}, the initialization is fixed by the reweighted maximal correlation method. The maximum iteration number of every inner nonconvex optimization for \texttt{IncrePR} is set to 3000, while the maximum number of iterations for \texttt{RAF+r} and \texttt{RAF} is three times as many. The average and minimum/maximum relative errors (in parenthesis separated by commas) are listed in Table~\ref{tab:2}. Relative reconstruction error is calculated by the method provided in the source code accompanying~\cite{Chandra2017PhasePack}. From Table~\ref{tab:2}, superiority of Algorithm~\ref{alg:2} and the fact that the more measurement data, the better the recovery are clear. The best reconstructions for each number of measurements over ten runs are plotted in Figure~\ref{fig:5}. The subcaption below each image is the relative reconstruction error. For $m=2n$ and $m=3n$, the reconstruction of Alg.2+r is much better. This is because \texttt{RAF} depends critically on the initialization which is sensitive to both measurement model and the number of measurements. 

\begin{figure}
\centering
     \begin{subfigure}[b]{.2\textwidth}
     \centering
     \includegraphics[width=\textwidth]{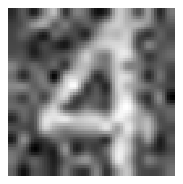}
     \caption{0.5107}
   \end{subfigure}
 \begin{subfigure}[b]{.2\textwidth}
     \centering
     \includegraphics[width=\textwidth]{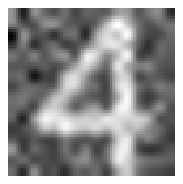}
     \caption{0.5054}
   \end{subfigure}
\begin{subfigure}[b]{.2\textwidth}
     \centering
     \includegraphics[width=\textwidth]{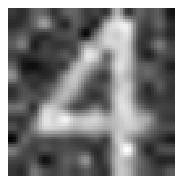}
     \caption{0.5079}
   \end{subfigure}
\begin{subfigure}[b]{.2\textwidth}
     \centering
     \includegraphics[width=\textwidth]{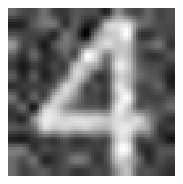}
     \caption{0.4958}
   \end{subfigure}\\
\begin{subfigure}[b]{.2\textwidth}
     \centering
     \includegraphics[width=\textwidth]{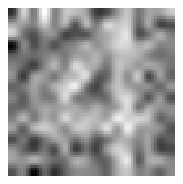}
     \caption{0.7532}
   \end{subfigure}
 \begin{subfigure}[b]{.2\textwidth}
     \centering
     \includegraphics[width=\textwidth]{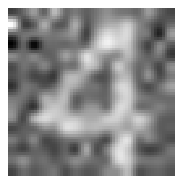}
     \caption{0.6770}
   \end{subfigure}
\begin{subfigure}[b]{.2\textwidth}
     \centering
     \includegraphics[width=\textwidth]{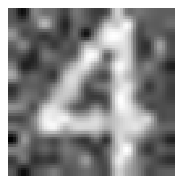}
     \caption{0.5904}
   \end{subfigure}
\begin{subfigure}[b]{.2\textwidth}
     \centering
     \includegraphics[width=\textwidth]{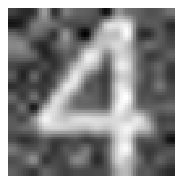}
     \caption{0.5524}
   \end{subfigure}\\
\begin{subfigure}[b]{.2\textwidth}
     \centering
     \includegraphics[width=\textwidth]{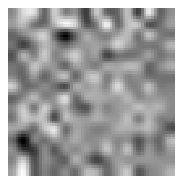}
     \caption{0.8222}
   \end{subfigure}
 \begin{subfigure}[b]{.2\textwidth}
     \centering
     \includegraphics[width=\textwidth]{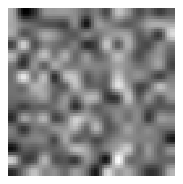}
     \caption{0.8189}
   \end{subfigure}
\begin{subfigure}[b]{.2\textwidth}
     \centering
     \includegraphics[width=\textwidth]{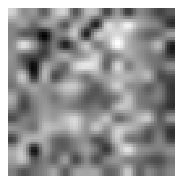}
     \caption{0.7724}
   \end{subfigure}
\begin{subfigure}[b]{.2\textwidth}
     \centering
     \includegraphics[width=\textwidth]{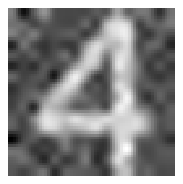}
     \caption{0.5491}
   \end{subfigure}
\caption{The best reconstructions by Algorithm~\ref{alg:2} + refinement (Alg.2+r), \texttt{RAF+r} and \texttt{RAF} over 10 runs from the number of measurements $m=2n,3n,4n,5n$. Top row are reconstructions by Alg.2+r, middle row are reconstructions by \texttt{RAF}, bottom row are reconstructions by \texttt{RAF+r}.\label{fig:5}}
\end{figure}

\subsection{Oversampling Fourier phase retrieval}
In this section, we consider the more difficult oversampling Fourier phase retrieval problem of reconstructing an image from its diffraction data. Unlike Gaussian model, where ambiguity of the solution only comes from a constant phase shift, there are three different factors that can cause ambiguity~\cite{Hayes1982}. All the state-of-the-art nonconvex solvers fail to find a even mediocre solution. Trace minimization convex problem~\eqref{eq:sdp} is no longer tight, and its direct solver also fails, as reported in~\cite{Candes2013}. Here we test the performance of \texttt{IncrePR} with restart for recovering the synthetic images of Cameraman and Barbara, all of size $128\times 128$, from oversampling data. If an image is of size $n_1\times n_2$, we pad the image with zeros to form an image of size $2n_1\times 2n_2$ and record the FFT intensities of the padded image. The data fidelity function is the Poisson maximum likelihood estimation~\eqref{eq:pfunc}, due to the inapplicability of intensity least-squares for oversampling Fourier phase retrieval~\cite{Li161}. For the noiseless Fourier phase retrieval, the measurement data determine the trace of the unknown lifting matrix, i.e., the Frobenius norm of the image. Hence, we just solve this problem by Algorithm~\ref{alg:2} (without the second step of solving~\eqref{eq:phaselift} with $\lambda=\lambda_0$) a few times. 

\begin{table}
\setlength{\tabcolsep}{.5\tabcolsep}
\centering
\caption{Relative error of \texttt{IncrePR} with ten repeats for oversampling Fourier phase retrieval.\label{tab:3}}
\scalebox{.85}{\begin{tabular}{llcccccccccc}
\toprule
repeat & $k$ & 1 & 2& 3& 4& 5& 6& 7& 8& 9& 10\\
\midrule
\multicolumn{12}{l}{Cameraman} \\
\midrule
\multirow{3}{*}{heuristic} &aver. &0.1107 &0.0671 &0.0549 &0.0476 &0.0434 &0.0405 &0.0398 &0.0378 &0.0364 &0.0361 \\
&min. &0.0944 &0.0531 &0.0436 &0.0395 &0.0380 &0.0352 &0.0356 &0.0319 &0.0309 &0.0311 \\
&max. & 0.1501 &0.0740 &0.0644 &0.0585 &0.0516 &0.0508 &0.0482 &0.0460 &0.0434 &0.0431\\
\midrule
\multirow{3}{*}{svd} &aver. &0.1126 &0.0679 &0.0553 &0.0480 &0.0449 &0.0425 &0.0403 &0.0388 &0.0379 &0.0371 \\
&min. & 0.0937 &0.0584 &0.0482 &0.0432 &0.0351 &0.0324 &0.0312 &0.0307 &0.0302 &0.0286 \\
&max. & 0.1503 &0.0739 &0.0638 &0.0542 &0.0522 &0.0475 &0.0443 &0.0437 &0.0414 &0.0436 \\
\midrule
\multicolumn{12}{l}{Barbara} \\
\midrule
\multirow{3}{*}{heuristic}&aver. & 0.0997 &0.0338 &0.0290 &0.0272 &0.0258 &0.0255 &0.0254 &0.0246 &0.0243 &0.0242 \\
&min. & 0.0465 &0.0314 &0.0259 &0.0232 &0.0214 &0.0208 &0.0199 &0.0204 &0.0204 &0.0194 \\
&max. &0.2358 &0.0381 &0.0344 &0.0307 &0.0305 &0.0282 &0.0288 &0.0285 &0.0278 &0.0274 \\
\midrule
\multirow{3}{*}{svd} &aver. &0.1012 &0.0561 &0.0328 &0.0271 &0.0255 &0.0249 &0.0240 &0.0241 &0.0236 &0.0236 \\
&min. & 0.0391 &0.0269 &0.0224 &0.0222 &0.0217 &0.0217 &0.0207 &0.0200 &0.0199 &0.0195 \\
&max. & 0.2800 &0.2753 &0.0793 &0.0369 &0.0321 &0.0308 &0.0301 &0.0308 &0.0291 &0.0274 \\
\bottomrule
\end{tabular}}
\end{table}

After obtaining the rank one approximation factor of current solution, starting from it, we can run \texttt{IncrePR} again. The number of repeat $K$ is set to $10$. The termination $\epsilon$ for each repeat is set to $0.1$. The initial elements of $\bm{Y}_0$ are integers randomly chosen from $0$ to $100$. To avoid large-scale SVD computations in our \texttt{IncrePR} algorithm to extract the optimal rank one approximation at the end of each repeat, we just choose the column with the maximum norm, which turns out to work as well as SVD approach for rank one approximation. We run the repeat procedure starting from $10$ random initial guesses for each image. The average, minimum and maximum relative errors (after rotation if necessary) over ten runs are listed in Table~\ref{tab:3}, where rank one approximation at each repeat using our heuristic way and SVD are compared. It shows that the quality of reconstruction is generally improved as the number of repeat $k$ increases. The recovered images from one run are displayed in Figure~\ref{fig:6}, where we only show results from the first three repeats of Algorithm~\ref{alg:2}. To demonstrate the failure of direct optimization for the nonconvex problem~\eqref{eq:nonconvex} with Poisson maximum likelihood estimation and fixed $p=1$, we show the results in subfigures (a) and (e) in Figure~\ref{fig:6}. To the best of our knowledge, it is the first nonconvex optimization based method that works well for the Fourier phase retrieval problem. The reconstruction quality is as good as the popular \texttt{HIO} method based on projection. Although \texttt{IncrePR} is more expensive than \texttt{HIO} in terms of computation cost, it has the flexibility of adding other regularization terms or prior information in the objective function.

\begin{figure}
\centering
\begin{subfigure}[b]{.2\textwidth}
     \centering
     \includegraphics[width=\textwidth]{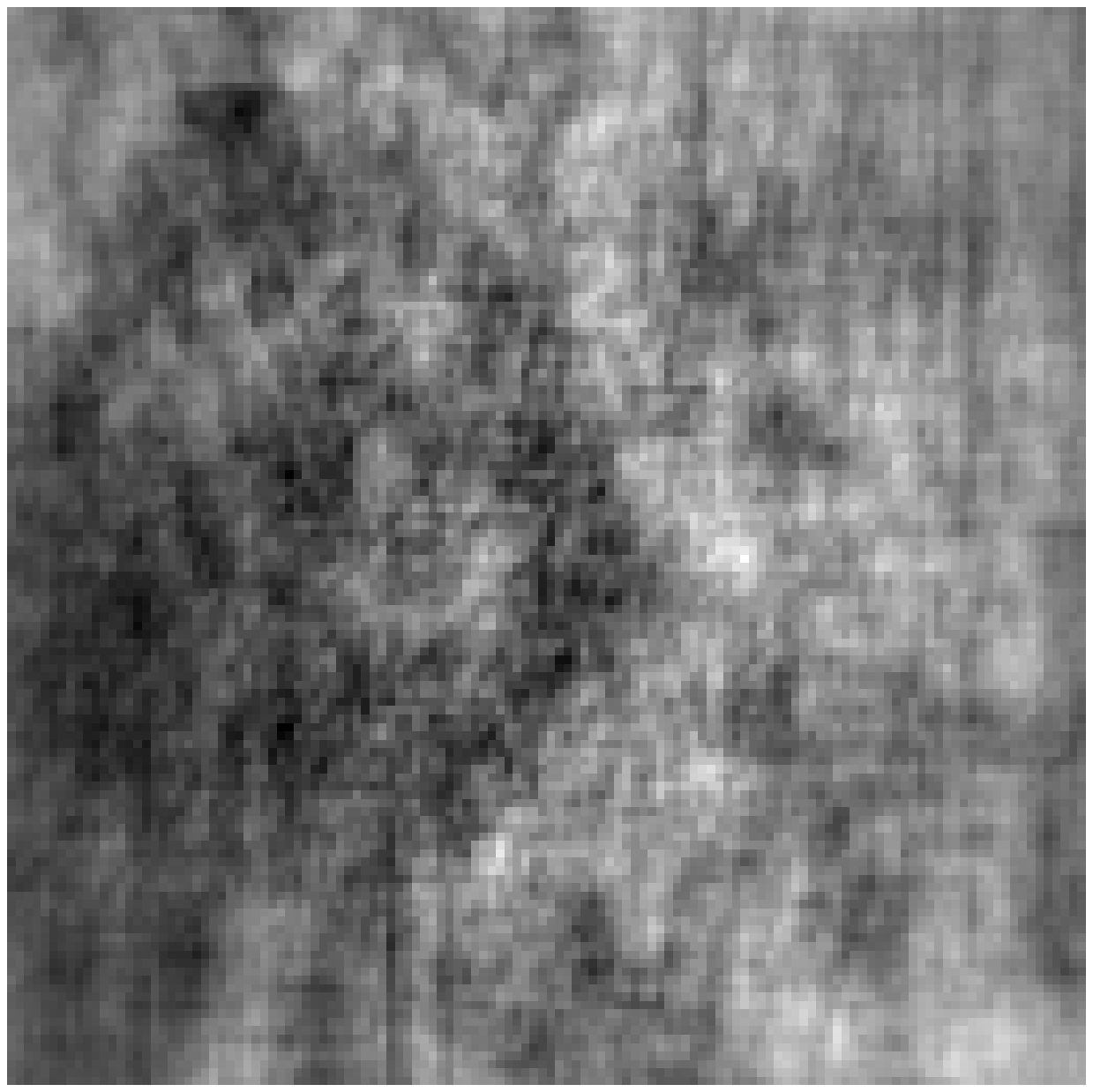}
     \caption{0.4044}
   \end{subfigure}
 \begin{subfigure}[b]{.2\textwidth}
     \centering
     \includegraphics[width=\textwidth]{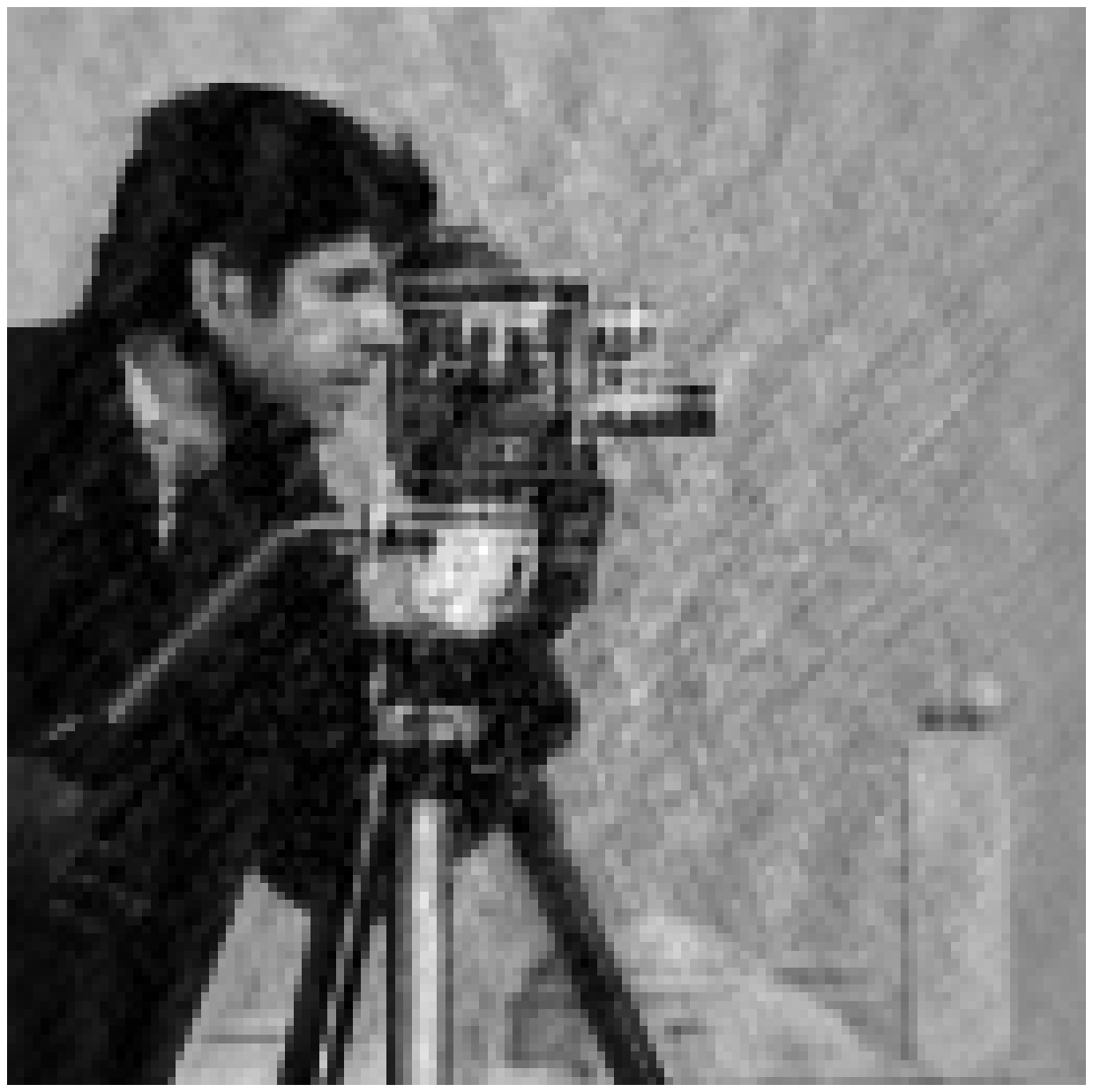}
     \caption{0.1032}
   \end{subfigure}
\begin{subfigure}[b]{.2\textwidth}
     \centering
     \includegraphics[width=\textwidth]{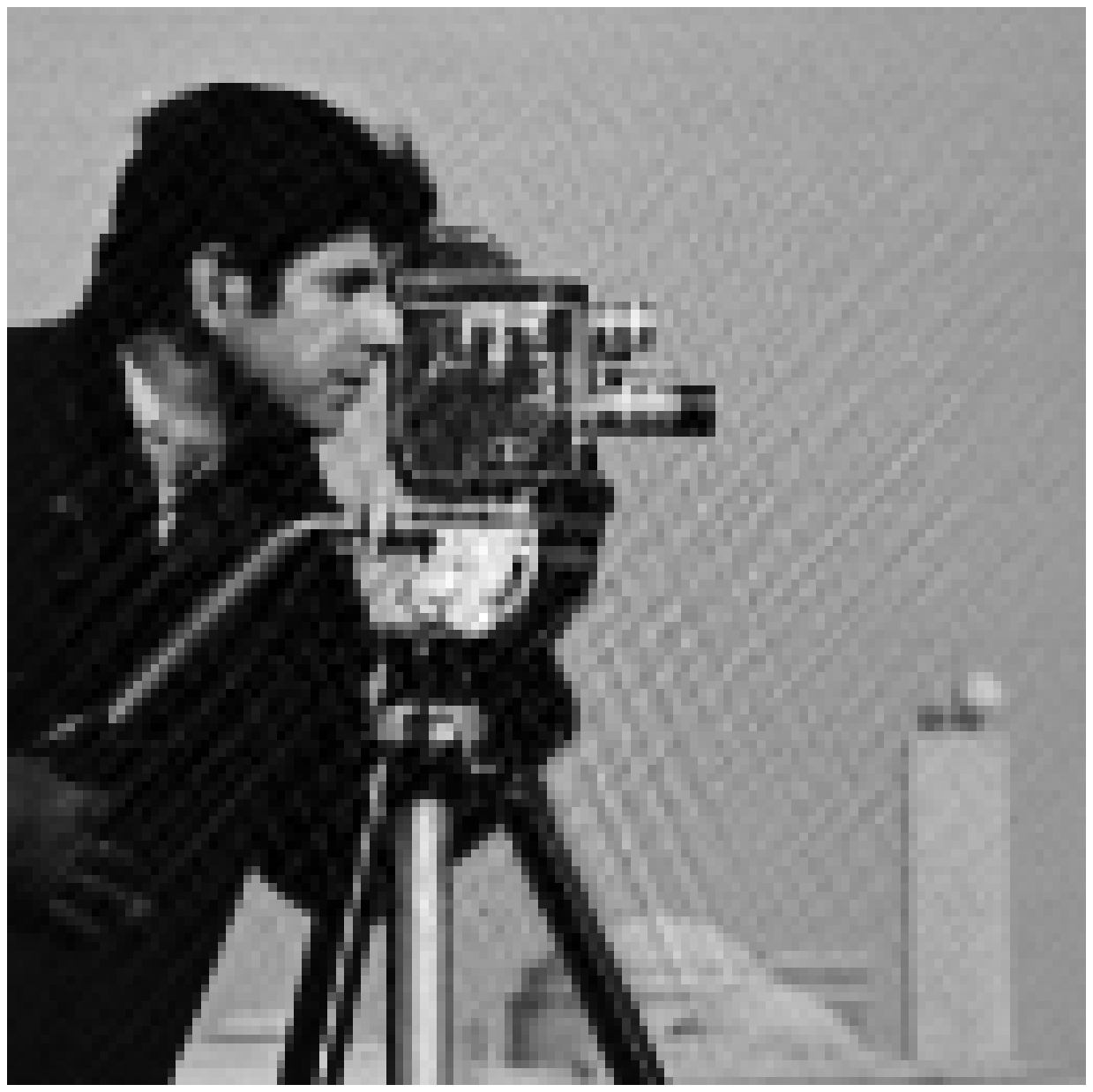}
     \caption{0.0531}
   \end{subfigure}
\begin{subfigure}[b]{.2\textwidth}
     \centering
     \includegraphics[width=\textwidth]{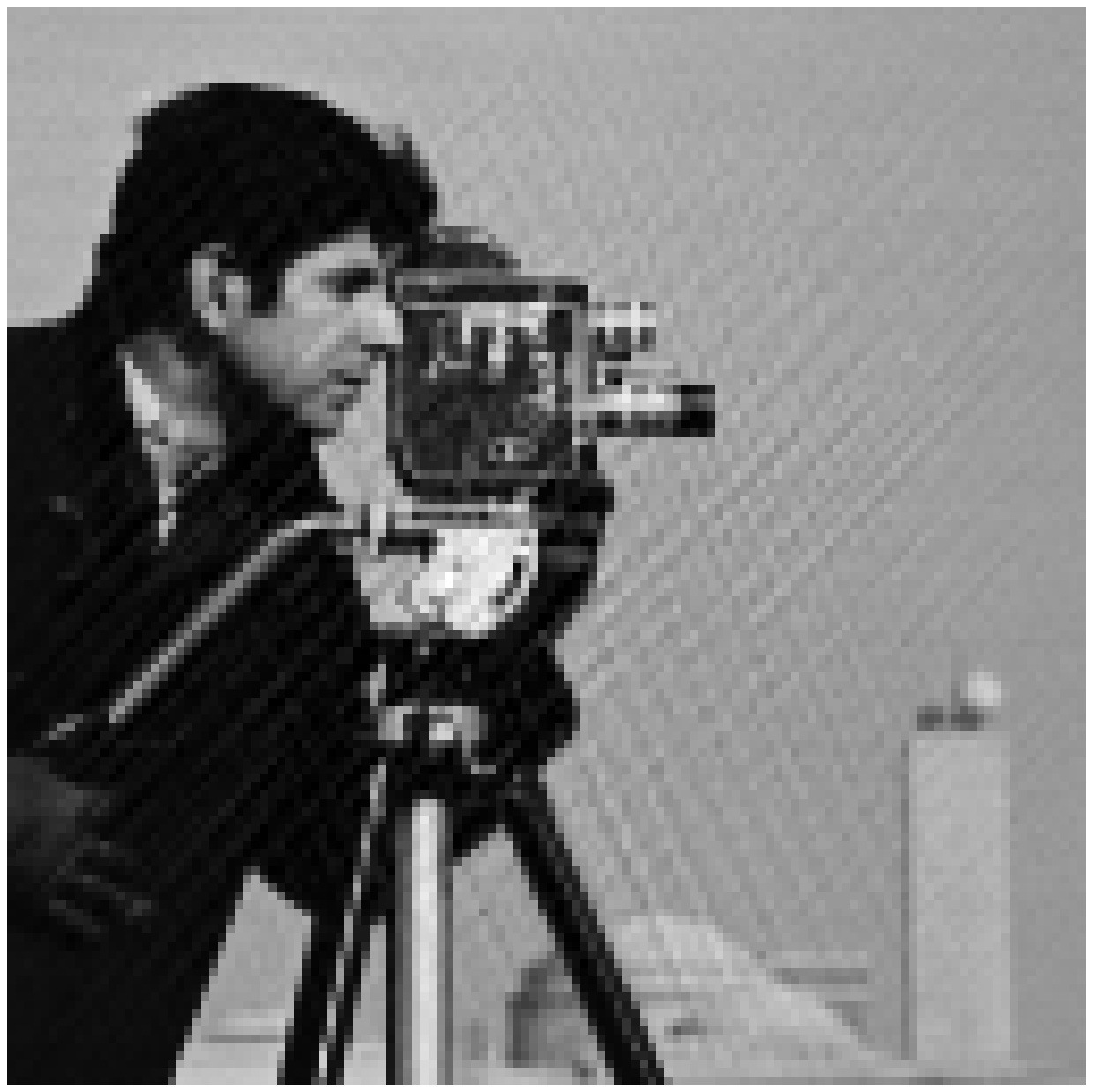}
     \caption{0.0436}
   \end{subfigure}\\
\begin{subfigure}[b]{.2\textwidth}
     \centering
     \includegraphics[width=\textwidth]{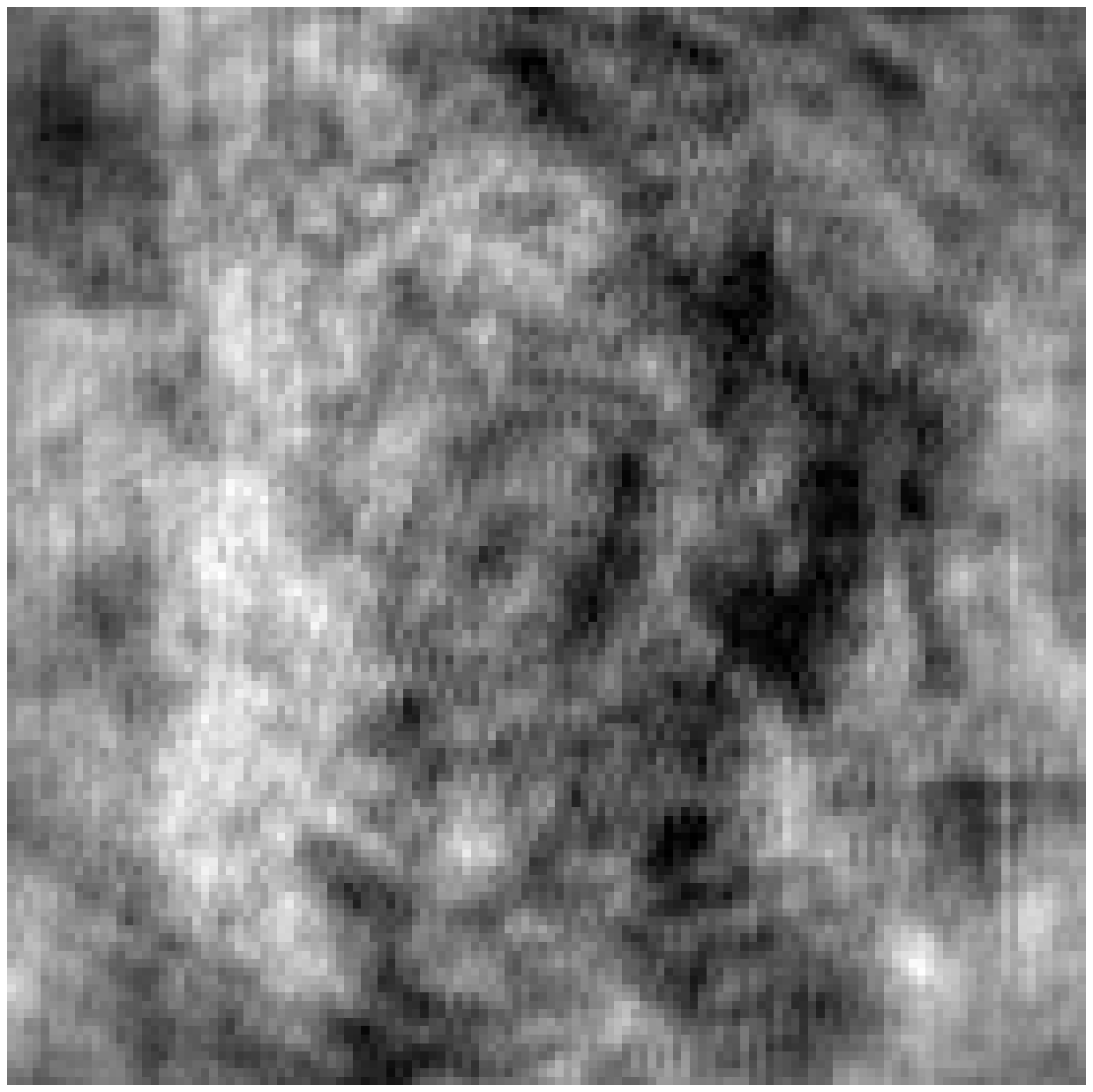}
     \caption{0.2831}
   \end{subfigure}
 \begin{subfigure}[b]{.2\textwidth}
     \centering
     \includegraphics[width=\textwidth]{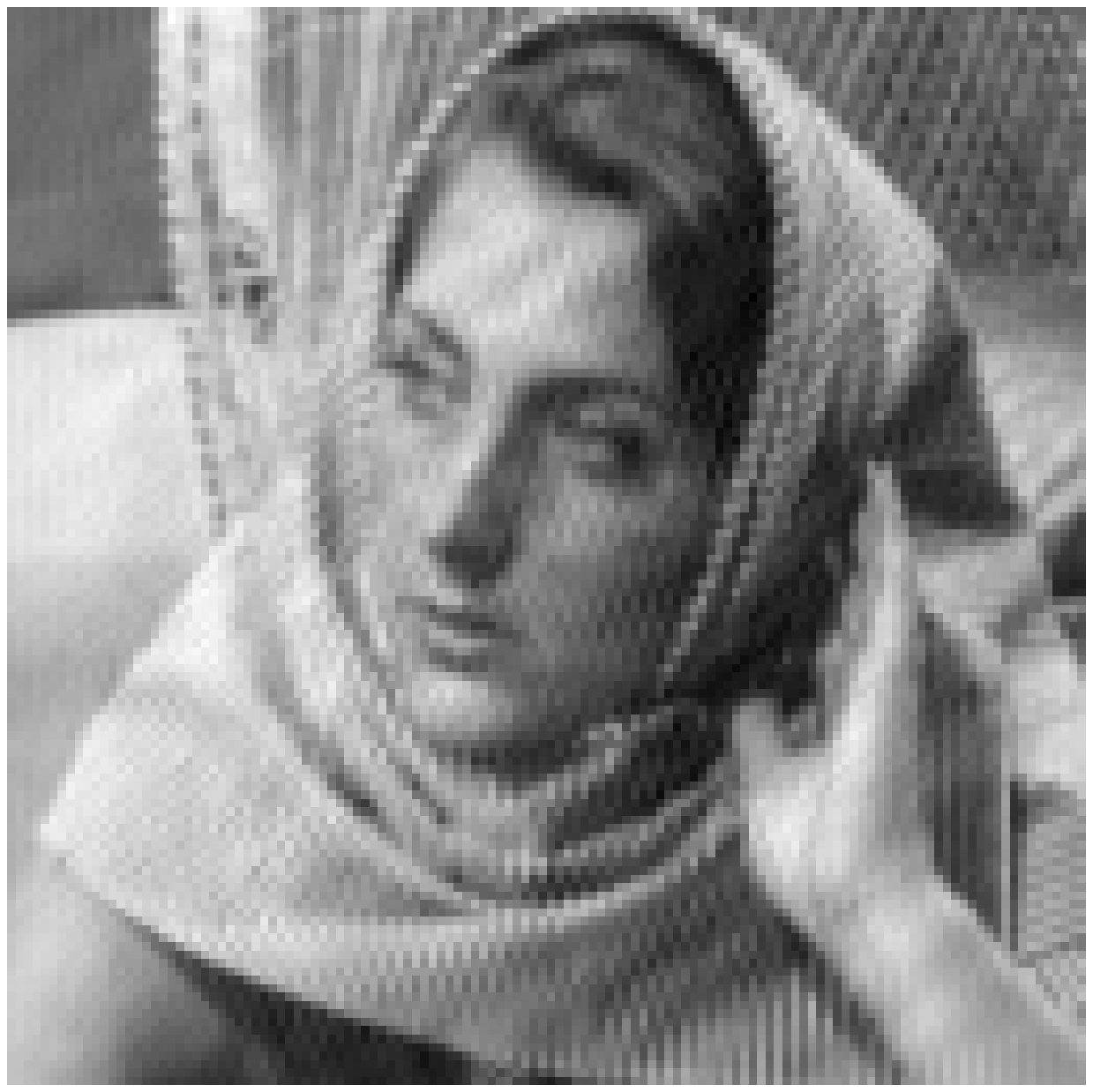}
     \caption{0.0524}
   \end{subfigure}
\begin{subfigure}[b]{.2\textwidth}
     \centering
     \includegraphics[width=\textwidth]{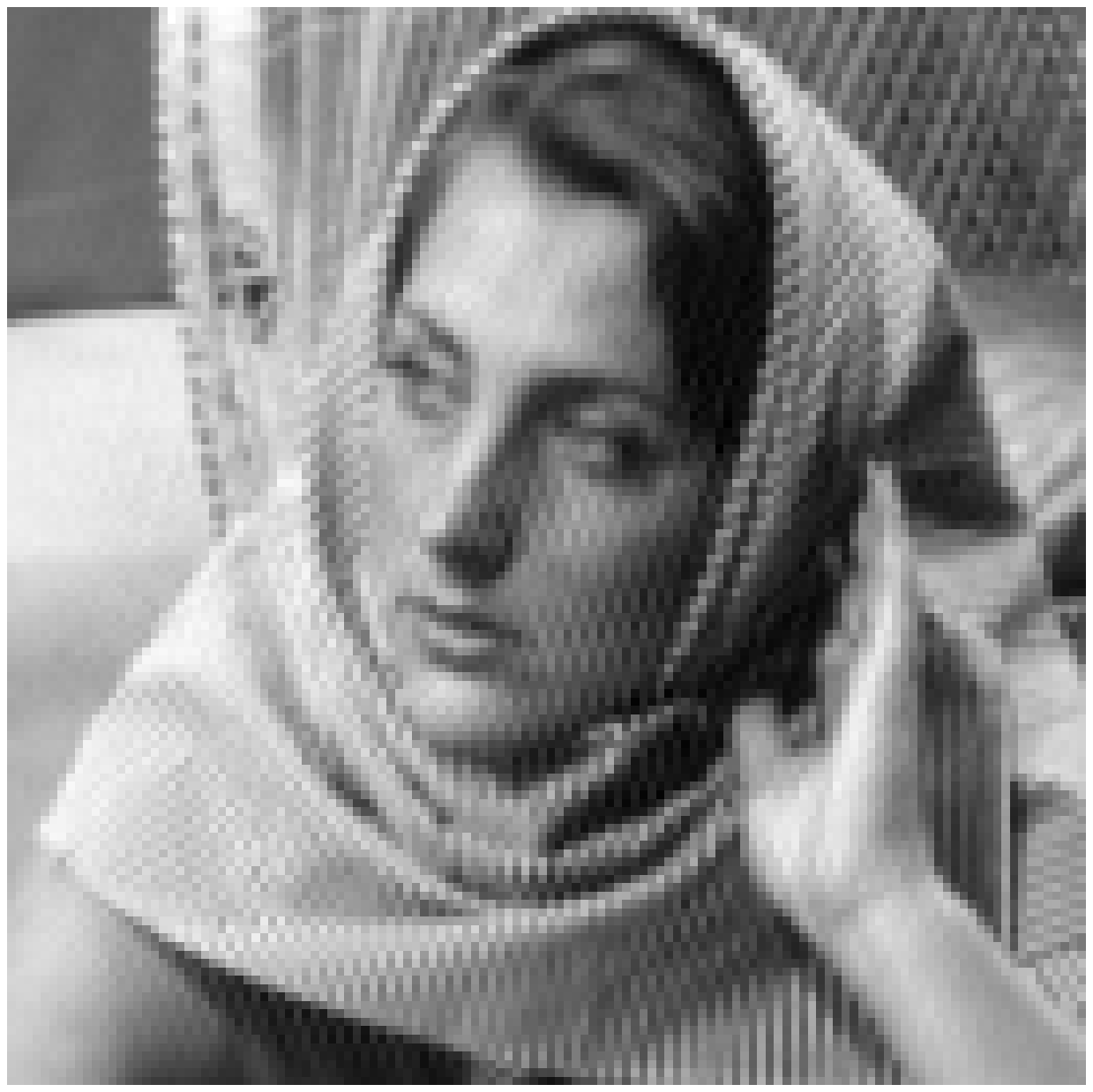}
     \caption{0.0330}
   \end{subfigure}
\begin{subfigure}[b]{.2\textwidth}
     \centering
     \includegraphics[width=\textwidth]{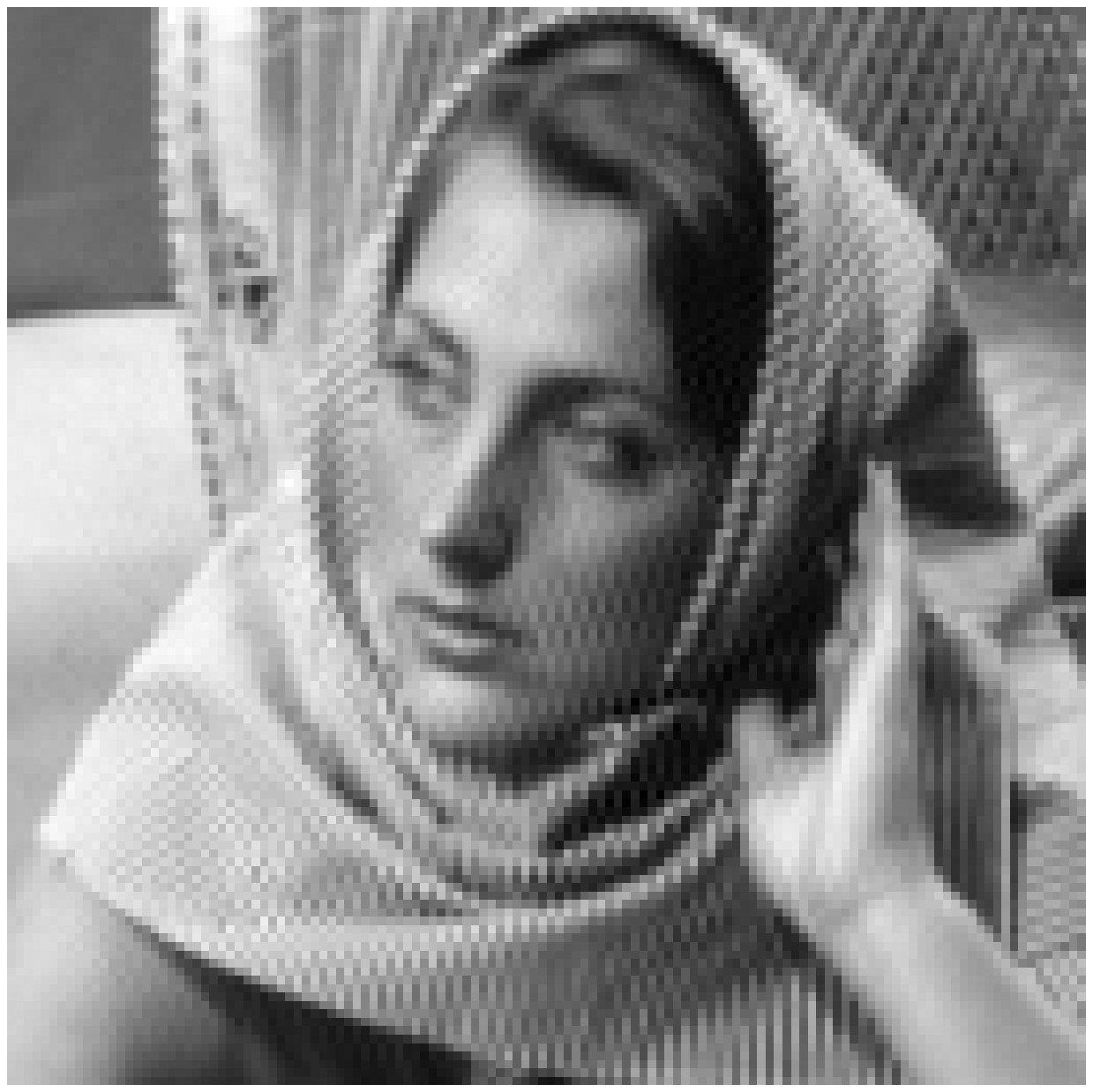}
     \caption{0.0259}
   \end{subfigure}
\caption{First row: reconstruction of Cameraman, (a) from direct nonconvex optimization of~\eqref{eq:nonconvex} with $p=1$, and (b) (c) (d) from \texttt{IncrePR} of the first three repeats. Second row: same tests for Barbara.\label{fig:6}}
\end{figure}

\section{Conclusion}
\label{sec:conclusion}
In this work, we propose an incremental nonconvex approach, \texttt{IncrePR}, to solve the phase retrieval problem based on the convex semidefinite relaxation (SDR) PhaseLift formulation. For Gaussian model, \texttt{IncrePR} with a restart strategy obtains the sharpest phase transition compared with the state-of-the-art nonconvex solvers. Since \texttt{IncrePR} solves the PhaseLift SDR problem, it avoids sensitive dependence of initialization and achieves global convergence, even when the number of measurements is close to the theoretical limit. It can escape stationary points and decrease the object function of PhaseLift by increasing the column number of the decomposition factor successively. Compared to standard  SDP solvers for PhaseLift, it reduces storage demand by matrix decomposition and improves computation efficiency by avoiding eigenvalue decomposition at each iteration. For phase retrieval for non-Gaussian models, such as transmission matrix and oversampling Fourier measurements,  although PhaseLift formulation is not tight and the optimal solution is not necessarily of rank one, one can still use \texttt{IncrePR} with restart to find a good approximation and then improve it further by a nonconvex solver, which often produces a better solution than that obtained by a direct nonconvex solver.

\section*{Acknowledgments}

JL was supported by China Postdoctoral Science Foundation grant No. 2017M620589. JFC was supported in part by Hong Kong Research Grant Council (HKRGC) grant 16306317. The work of Hongkai Zhao is partially supported by the summer visiting program at CSRC.


\end{document}